\documentclass[12pt]{amsart}

\usepackage{amsmath}
\usepackage{amsthm}
\usepackage{amssymb}
\usepackage{amsfonts}
\usepackage{latexsym, graphicx}
\usepackage[top=1in, bottom=1.5in, left=1.5in, right=1.5in]{geometry}
\usepackage{csquotes}

\usepackage{mathtools}

\usepackage{tikz}
\usetikzlibrary{shapes.multipart,shapes,calc,decorations.pathreplacing}

\usepackage{varwidth}
\usepackage{ytableau}

\usepackage{thmtools}
\usepackage{thm-restate} 

\newtheorem{thm}{Theorem}[section]
\newtheorem{lemma}[thm]{Lemma}
\newtheorem{prop}[thm]{Proposition}
\newtheorem{defn}[thm]{Definition}

\newtheorem{cor}[thm]{Corollary}

\newtheorem{conj}[thm]{Conjecture}

\theoremstyle{remark}
\newtheorem{remark}[thm]{Remark}

\newtheorem{ex}[thm]{Example}

\newcommand{\wt}{\widetilde}
\newcommand{\wh}{\widehat}

\newcommand{\cA}{\mathcal{A}}

\newcommand{\cE}{\mathcal{E}}

\newcommand{\rk}{\operatorname{rk}}

\newcommand{\Add}{\operatorname{Add}}
\newcommand{\LHS}{\operatorname{LHS}}
\newcommand{\RHS}{\operatorname{RHS}}

\newcommand{\boat}[1][]{%
\def \a {0.84};
\draw[fill = white] ({22*\a/405},{61*\a/405}) -- ({92*\a/405},{-32*\a/405}) -- ({312*\a/405},{-32*\a/405}) -- ({382*\a/405},{61*\a/405}) -- ({22*\a/405},{61*\a/405});
\draw[thick] ({202*\a/405}, {61*\a/405}) -- ({202*\a/405},{373*\a/405});
\draw[fill = white] ({202*\a/405},{373*\a/405}) -- ({362*\a/405},{97*\a/405}) -- ({202*\a/405},{97*\a/405});
}

\title{Reconstructing Partitions from their Multisets of $k$-Minors}
\author{Pakawut Jiradilok}\address{Department of Mathematics, Harvard University, One Oxford Street, Cambridge, MA 02138}\email{pjiradilok@college.harvard.edu}

\begin{document}

\begin{abstract}
For non-negative integers $n$ and $k$ with $n \ge k$, a {\em $k$-minor} of a partition $\lambda = [\lambda_1, \lambda_2, \dots]$ of $n$ is a partition $\mu = [\mu_1, \mu_2, \dots]$ of $n-k$ such that $\mu_i \le \lambda_i$ for all $i$. The multiset $\wh{M}_k(\lambda)$ of $k$-minors of $\lambda$ is defined as the multiset of $k$-minors $\mu$ with multiplicity of $\mu$ equal to the number of standard Young tableaux of skew shape $\lambda / \mu$. We show that there exists a function $G(n)$ such that the partitions of $n$ can be reconstructed from their multisets of $k$-minors if and only if $k \le G(n)$. Furthermore, we prove that $\lim_{n \rightarrow \infty} G(n)/n = 1$ with $n-G(n) = O(n/\log n)$. As a direct consequence of this result, the irreducible representations of the symmetric group $S_n$ can be reconstructed from their restrictions to $S_{n-k}$ if and only if $k \le G(n)$ for the same function $G(n)$. For a minor $\mu$ of the partition $\lambda$, we study the excitation factor $E_\mu (\lambda)$, which appears as a crucial part in Naruse's Skew-Shape Hook Length Formula. We observe that certain excitation factors of $\lambda$ can be expressed as a $\mathbb{Q}[k]$-linear combination of the elementary symmetric polynomials of the hook lengths in the first row of $\lambda$ where $k = \lambda_1$ is the number of cells in the first row of $\lambda$.
\end{abstract}

\maketitle

\tableofcontents

\section{Introduction}
Given non-negative integers $n$ and $k$ with $n \ge k$, a {\em $k$-minor} of a partition $\lambda = [\lambda_1, \lambda_2, \dots]$ of $n$ is a partition $\mu = [\mu_1, \mu_2, \dots]$ of $n-k$ such that $\mu_i \le \lambda_i$ for all $i$. For each partition $\lambda$ of $n$, the multiset $\wh{M}_k(\lambda)$ of $k$-minors of $\lambda$ is defined as the multiset of $k$-minors $\mu$ of $\lambda$ with multiplicity of $\mu$ equal to the number $N(\lambda/\mu)$ of standard Young tableaux of skew shape $\lambda / \mu$. This paper considers the question of whether all the partitions of $n$ can be reconstructed from their multisets of $k$-minors for each given $(n,k)$. When the reconstruction is possible, we shall say that multiset-reconstructibility (MRC) holds for the pair $(n,k)$. We prove that there exists a function $G(n)$ such that MRC holds for $(n,k)$ if and only if $k \le G(n)$. In Theorem \ref{thm:Gnsimn}, we show that 
\[
\lim_{n \rightarrow \infty} \frac{G(n)}{n} = 1
\]
with $n-G(n) = O(n/\log n)$.

This partition multiset-reconstruction problem is a natural variant to the partition reconstruction problem. Instead of the multiset $\wh{M}_k(\lambda)$, if we ignore the multiplicities, we obtain the {\em set of $k$-minors} $M_k(\lambda)$ of $\lambda$. The partition reconstruction problem asks for which $(n,k)$ the partitions of $n$ can be reconstructed from their sets of $k$-minors. For such a pair $(n,k)$, we say that reconstructibility (RC) holds. This problem was studied by Pretzel and Siemons \cite{PS05}, Vatter \cite{Vat08}, and Monks \cite{Mo09}. In 2009, Monks gave an explicit solution to this problem: RC holds for $(n,k)$ if and only if $k \le g(n)$ for an explicitly described number-theoretic function $g(n)$ with 
\[
\sqrt{n+2} - 2 \le g(n) \le \sqrt{n+2} + 3 \sqrt[4]{n+2}.
\]
Observe that Monks' function $g(n)$ satisfies the asymptotic property $g(n) \sim \sqrt{n}$ as $n \rightarrow \infty$. For multiset-reconstructibility, our function $G(n)$ is analogous to Monks' function $g(n)$ in the sense that both are the threshold values of $k$ for which reconstructibility holds. Nevertheless, by adding the data of multiplicities to $M_k(\lambda)$, we find that the analogous function $G(n)$ grows significantly faster than $g(n)$ does, as $G(n) \sim n$.

The partition multiset-reconstruction problem arises naturally in the representation theory of the symmetric group. Each partition $\lambda$ of $n$ corresponds to the irreducible representation $V_{\lambda}$ of the symmetric group $S_n$. (See for example, \cite{Ful97, JK81, Sa13, Stan99}.) The data of the multiset $\wh{M}_k(\lambda)$ is precisely the data of decomposition of the representation $V_{\lambda}$ when we restrict $S_n$ to the Young subgroup $S_{n-k} \times S_1 \times S_1 \times \dots \times S_1 \subseteq S_n$. By restricting $S_n$ to $S_{n-k}$, we obtain a representation $V_\lambda$ of $S_{n-k}$ from the representation $V_{\lambda}$ of $S_n$. This new representation is, in general, not irreducible. Repeated applications of the Littlewood-Richardson rule show that $V_\lambda$ decomposes as a direct sum of irreducible representations of $n-k$:
\[
V_{\lambda} \cong \bigoplus_{\mu \, \vdash (n-k)} V_{\mu}^{\oplus N(\lambda/\mu)}.
\]
From this point of view, the partition multiset-reconstruction problem asks whether the decomposition on the right hand side of the above equation is sufficient for recovering the original irreducible representation $V_{\lambda}$, for each $(n,k)$. Our result implies that the irreducible representations of $S_n$ can be reconstructed from their restrictions to $S_{n-k}$ if and only if $k \le G(n)$ for a function $G$ with $\lim_{n \rightarrow \infty} G(n)/n = 1$.

A convenient tool we use in our work is Naruse's Skew-Shaped Hook Length Formula from a work of Naruse's \cite{Na14} in 2014. Whereas the renowned Hook Length Formula expresses the number $N(\lambda)$ of Standard Young Tableaux (SYT) of straight shape $\lambda$ via the product of hook lengths in $\lambda$, Naruse's formula expresses the number $N(\lambda/\mu)$ of Standard Young Tableaux (SYT) of {\em skew} shape $\lambda / \mu$ via the sum of the products of hook lengths in what Naruse calls the {\em excited diagrams} of $\mu$ in $\lambda$. We call this sum the {\em excitation factor} of $\mu$ in $\lambda$ and show in Theorem \ref{thm:Qklincomb} that when $\mu = [m]$ is a partition of one part of size $m$, the excitation factor $E_{\mu}(\lambda)$ can be expressed as a $\mathbb{Q}[k]$-linear combination of elementary symmetric polynomials $\sigma_i(a_1, \dots, a_k)$, where $a_1, \dots, a_k$ are the hook lengths in the first row of $\lambda$. We also give an explicit formula for the coefficients of $\sigma_i$ in $E_{\mu}(\lambda)$. Interested readers can find more details about Naruse's Skew Shape Hook Length Formula in \cite{MPP15, Na14}.

The formula in Theorem \ref{thm:Qklincomb} allows us to prove Theorem \ref{thm:SONAR}. The latter theorem serves as a technique for recovering a partition, which we call the \enquote{SONAR} technique. Our main result, Theorem \ref{thm:Gnsimn}, is proved by using SONAR to give a lower bound to $G(n)$.

Section \ref{sec:n-G(n)} is devoted to explicit calculations and upper bounds for the function $G(n)$. In Propositions \ref{p:Gnisn-2} and \ref{p:Gnisn-3}, we prove that $n-G(n) = 2$ if and only if $2 \le n \le 11$ or $n = 13$, and that $n-G(n) = 3$ if and only if $n \in \{12,14,17,18,23\}$. We give some computational results for known values of $G(n)$.

In the final section, we suggest a few open questions and ideas related to multiset-reconstructibility of partitions which give directions for further investigation.

\bigskip

\section{Notations and Definitions}
This section is devoted to providing the basic notations and definitions we use throughout the paper. The set of all positive integers is denoted $\mathbb{Z}_{>0}$, while the set of all non-negative integers is denoted $\mathbb{Z}_{\ge 0}$. In this paper, the notation $\log$ denotes the natural logarithm function.

For each polynomial $P(X) \in \mathbb{C}[X]$ (or a generating function $P(X) \in \mathbb{C}[[X]]$) and an integer $i \in \mathbb{Z}_{\ge 0}$, we use the notation $[P(X)]_{X^i}$ to denote the coefficient of $X^i$ in $P(X)$.

\begin{defn}
A {\em partition} $\lambda$ of $n \in \mathbb{Z}_{\ge 0}$ is an array $\lambda = [\lambda_1, \lambda_2, \dots]$ of non-negative integers with $\lambda_1 \ge \lambda_2 \ge \cdots$ such that $\sum_{i=1}^\infty \lambda_i = n$. We call $n$ the {\em size} of $\lambda$ and write $|\lambda| = n$. We also use the notation $\lambda \vdash n$ to mean that $\lambda$ is a partition of $n$.
\end{defn}

We usually truncate the zeroes in the partition notation, so that the notation $[\lambda_1, \lambda_2, \dots, \lambda_m]$ refers to the partition $[\lambda_1, \lambda_2, \dots, \lambda_m, 0, 0, 0, \dots]$.

It is useful to think of partitions as Young diagrams. A {\em Young diagram} corresponding to the partition $\lambda \vdash n$ is a collection of $n$ cells, arranged in the left-justified manner, with exactly $\lambda_i$ cells on the $i$-th row. The number of cells in the $j$-th column is therefore $\#\{i: \lambda_i \ge j\}$.

\begin{defn}
Let $\lambda$ be a partition of a non-negative integer $n$. The {\em conjugate} $\lambda^t$ of $\lambda$ is a partition of $n$ whose number $(\lambda^t)_i$ of cells in the $i$-th row is the number of cells in the $i$-th column of $\lambda$, for every positive integer $i$.

A partition $\lambda$ is called {\em self-conjugate} if $\lambda = \lambda^t$.
\end{defn}

\begin{defn}
Let $n$ and $k$ be non-negative integers with $n \ge k$. Let $\lambda = [\lambda_1, \lambda_2, \dots]$ be a partition of $n$. A {\em $k$-minor} $\mu$ of $\lambda$ is a partition $\mu = [\mu_1, \mu_2, \dots]$ of the non-negative integer $n-k$ such that $\mu_i \le \lambda_i$ for all $i$. We use the notation $\mu \le \lambda$ to mean that $\mu$ is a minor of $\lambda$.
\end{defn}

For example, the partition $[4,1]$ is a $4$-minor of the partition $[4,3,2]$. We can see this relationship in Figure \ref{fig:41in432}. 

\begin{center}
\begin{figure}
\begin{tikzpicture}
\ytableausetup{notabloids}
\ytableausetup{mathmode, boxsize=2.0em}
\node (n) {\ytableausetup{nosmalltableaux}
\ytableausetup{notabloids}
\ydiagram[*(gray)]{4,1}
*[*(white)]{0,1+2,2}};
\end{tikzpicture}
\caption{The partition $[4,1]$ is a $4$-minor of the partition $[4,3,2]$.} \label{fig:41in432}
\end{figure}
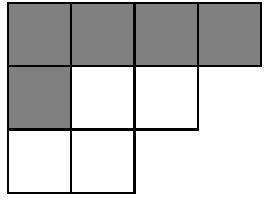
\end{center}

For a given partition $\lambda$ of a positive integer $n$, we write $\wh{M}_k(\lambda)$ to denote the multiset of $k$-minors of $\lambda$ with multiplicity of each $k$-minor $\mu$ equal to the number of standard Young tableaux of skew shape $\lambda / \mu$. 

For a partition $\lambda$, the number of standard Young tableaux of shape $\lambda$ is denoted by $N(\lambda)$. For partitions $\mu$ and $\lambda$ with $\mu \le \lambda$, the number of standard Young tableaux of skew shape $\lambda / \mu$ is denoted by $N(\lambda / \mu)$. By convention, if $\mu$ is not a minor of $\lambda$, then $N(\lambda/\mu) = 0$. Also, by convention, $N([0]) = 1$ where $[0]$ is the empty partition. For example, the $4$-minor $[4,1]$ in the multiset $\wh{M}_4([4,3,2])$ has multiplicity $5$, because there are exactly five standard Young tableaux of skew shape $[4,3,2]/[4,1]$ as shown in Figure \ref{fig:STY432/41}.

\begin{center}
\begin{figure}
\begin{tikzpicture}
\begin{scope}[shift = {(-5.0,0)}]
\ytableausetup{notabloids}
\ytableausetup{mathmode, boxsize=2.0em}
\node (n) {\ytableausetup{nosmalltableaux}
\ytableausetup{notabloids}
\ydiagram[*(gray)]{4,1}
*[*(white)]{0,1+2,2}};

\def \a {0.84};

\node at ({-4.5*\a +3.5},{0.5*\a-0.425}) {$1$};
\node at ({-3.5*\a +3.5},{0.5*\a-0.425}) {$2$};
\node at ({-5.5*\a +3.5},{-0.5*\a-0.425}) {$3$};
\node at ({-4.5*\a +3.5},{-0.5*\a-0.425}) {$4$};
\end{scope}

\begin{scope}[shift = {(0.0,0)}]
\ytableausetup{notabloids}
\ytableausetup{mathmode, boxsize=2.0em}
\node (n) {\ytableausetup{nosmalltableaux}
\ytableausetup{notabloids}
\ydiagram[*(gray)]{4,1}
*[*(white)]{0,1+2,2}};

\def \a {0.84};

\node at ({-4.5*\a +3.5},{0.5*\a-0.425}) {$1$};
\node at ({-3.5*\a +3.5},{0.5*\a-0.425}) {$3$};
\node at ({-5.5*\a +3.5},{-0.5*\a-0.425}) {$2$};
\node at ({-4.5*\a +3.5},{-0.5*\a-0.425}) {$4$};
\end{scope}

\begin{scope}[shift = {(+5.0,0)}]
\ytableausetup{notabloids}
\ytableausetup{mathmode, boxsize=2.0em}
\node (n) {\ytableausetup{nosmalltableaux}
\ytableausetup{notabloids}
\ydiagram[*(gray)]{4,1}
*[*(white)]{0,1+2,2}};

\def \a {0.84};

\node at ({-4.5*\a +3.5},{0.5*\a-0.425}) {$1$};
\node at ({-3.5*\a +3.5},{0.5*\a-0.425}) {$4$};
\node at ({-5.5*\a +3.5},{-0.5*\a-0.425}) {$2$};
\node at ({-4.5*\a +3.5},{-0.5*\a-0.425}) {$3$};
\end{scope}

\begin{scope}[shift = {(-5.0,-4.0)}]
\ytableausetup{notabloids}
\ytableausetup{mathmode, boxsize=2.0em}
\node (n) {\ytableausetup{nosmalltableaux}
\ytableausetup{notabloids}
\ydiagram[*(gray)]{4,1}
*[*(white)]{0,1+2,2}};

\def \a {0.84};

\node at ({-4.5*\a +3.5},{0.5*\a-0.425}) {$2$};
\node at ({-3.5*\a +3.5},{0.5*\a-0.425}) {$3$};
\node at ({-5.5*\a +3.5},{-0.5*\a-0.425}) {$1$};
\node at ({-4.5*\a +3.5},{-0.5*\a-0.425}) {$4$};
\end{scope}

\begin{scope}[shift = {(0.0,-4.0)}]
\ytableausetup{notabloids}
\ytableausetup{mathmode, boxsize=2.0em}
\node (n) {\ytableausetup{nosmalltableaux}
\ytableausetup{notabloids}
\ydiagram[*(gray)]{4,1}
*[*(white)]{0,1+2,2}};

\def \a {0.84};

\node at ({-4.5*\a +3.5},{0.5*\a-0.425}) {$2$};
\node at ({-3.5*\a +3.5},{0.5*\a-0.425}) {$4$};
\node at ({-5.5*\a +3.5},{-0.5*\a-0.425}) {$1$};
\node at ({-4.5*\a +3.5},{-0.5*\a-0.425}) {$3$};
\end{scope}
\end{tikzpicture}
\caption{The five standard Young tableaux of skew shape $[4,3,2]/[4,1]$.} \label{fig:STY432/41}
\end{figure}
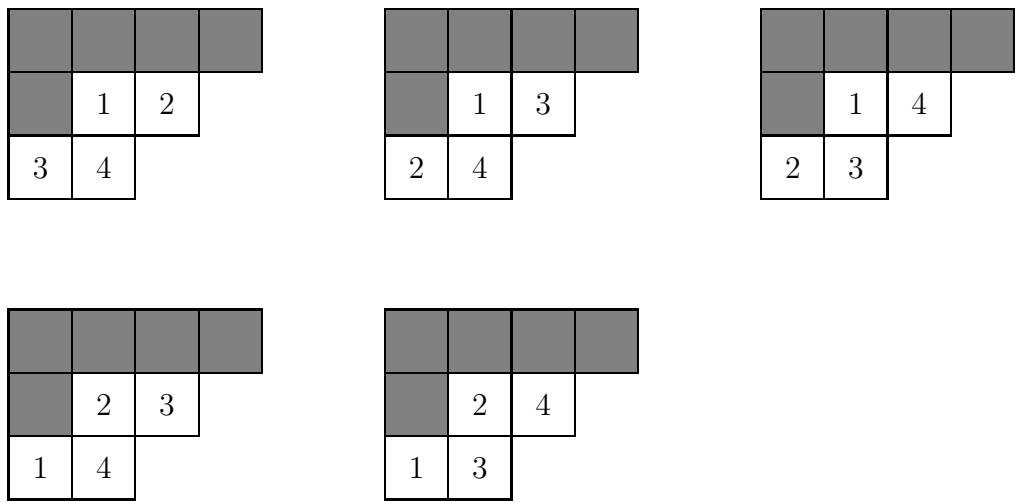
\end{center}

This paper focuses on the partition multiset-reconstructibility (MRC) Problem, which we state below.

\smallskip

\textbf{Multiset-Reconstructibility (MRC) Problem.} {\em 
For which pair $(n,k)$ of non-negative integers with $n \ge k$ does $\wh{M}_k(\mu) = \wh{M}_k(\nu)$ imply $\mu = \nu$ for all partitions $\mu$ and $\nu$ of $n$?}

\smallskip

When multiset-reconstructibility (MRC) is possible for $(n,k)$, we say that MRC holds for $(n,k)$.

Let $\mu$ and $\nu$ be partitions of a non-negative integer $n$. By definition, $\wh{M}_k(\mu) = \wh{M}_k(\nu)$ means $N(\mu/\tau) = N(\nu/\tau)$ for all partitions $\tau$ of $n-k$.

\begin{prop} \label{p:kthenk'}
Let $n \ge k' > k$ be non-negative integers. Let $\mu$ and $\nu$ be partitions of $n$ such that $\wh{M}_k(\mu) = \wh{M}_k(\nu)$. Then, $\wh{M}_{k'}(\mu) = \wh{M}_{k'}(\nu)$.
\end{prop}
\begin{proof}
For each partition $\pi$ of any positive integer, let $\Add(\pi)$ denote the set of all partitions which have $\pi$ as a $1$-minor. Let $\sigma$ be any partition of $n-k-1$. We then have
\[
N(\mu/ \sigma) = \sum_{\tau \in \Add(\sigma)} N(\mu/\tau) = \sum_{\tau \in \Add(\sigma)} N(\nu/\tau) = N(\nu/\sigma).
\]
Therefore, $\wh{M}_{k+1}(\mu) = \wh{M}_{k+1}(\nu)$ and by iteration we have $\wh{M}_{k'}(\mu) = \wh{M}_{k'}(\nu)$.
\end{proof}

The existence of the function $G(n)$ is justified by the following corollary of Proposition \ref{p:kthenk'}.

\begin{cor}
There exists a function $G(n)$ with the property that MRC holds for $(n,k)$ if and only if $k \le G(n)$.
\label{cor:Gexists}
\end{cor}

\begin{proof}
Let $n$, $k$, and $k'$ be non-negative integers with $n \ge k' > k$. Suppose that MRC fails for $(n,k)$. Then, there are two different partitions $\mu$ and $\nu$ of $n$ for which $\wh{M}_k(\mu) = \wh{M}_k(\nu)$. Proposition \ref{p:kthenk'} implies that MRC also fails for $(n,k')$ for every $n \ge k'>k$. 

MRC trivially holds when $k=0$ for every non-negative integer $n$. Therefore, we can define $G(n)$ to be the maximum integer $k \le n$ for which MRC holds for $(n,k)$. The argument above shows that MRC holds if and only if $k \le G(n)$.
\end{proof}

\bigskip

\section{Skew-Shape Hook Length Formula}
The famous hook length formula, which was discovered by Frame, Robinson, and Thrall \cite{FRT54} in 1954, states that, for every partition $\lambda$ of a non-negative integer $n$, the number of Standard Young Tableaux of shape $\lambda$ is
\[
N(\lambda) = \frac{n!}{H_{\lambda}}
\]
where $H_{\lambda}$ is the product of the hook lengths in $\lambda$.

In 2014, Naruse \cite{Na14} generalizes this formula to skew shapes. For partitions $\mu$ and $\lambda$ with $\mu \le \lambda$, Naruse's Skew-Shape Hook Length Formula expresses $N(\lambda/\mu)$ via the sum of the products of hook lengths in the excited diagrams of $\mu$ in $\lambda$. We will describe this formula in the following.

Let $\mu$ be a minor of $\lambda$. We use the matrix notation for labeling the cells, so that the cell $(i,j)$ in the Young diagram of $\lambda$ refers to the cell in the $i$-th row and the $j$-th column. Thus, the cells that appear in the Young diagram of $\lambda$ are labeled by $(i,j)$ such that $j \le \lambda_i$. Considering the Young diagram of $\mu$ inside the Young diagram of $\lambda$, we place a pebble in each of the cells inside the Young diagram of $\mu$. A {\em local excitation move} is defined as the operation of moving a pebble in the cell $(i,j)$ to $(i+1, j+1)$, if initially the three cells $(i,j+1)$, $(i+1,j)$, and $(i+1,j+1)$ are all unoccupied, and the cell $(i+1,j+1)$ is still inside $\lambda$. (See Figure \ref{fig:lem}.)

\begin{center}
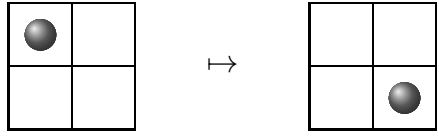
\begin{figure}
\begin{tikzpicture}
\begin{scope}[shift = {(-2.0,0)}]
\ytableausetup{notabloids}
\ytableausetup{mathmode, boxsize=2.0em}
\node (n) {\ytableausetup{nosmalltableaux}
\ytableausetup{notabloids}
\ydiagram[*(white)]{2,2}};

\def \a {0.84};
\def \r {\a/4};

\shade[shading=ball, ball color = gray] ({-4.5*\a +3.5},{1.0*\a-0.425}) circle (\r);
\end{scope}

\begin{scope}[shift = {(0,0)}]
\def \a {0.84};
\node at ({-4.0*\a +3.5},{0.5*\a-0.425}) {$\mapsto$};
\end{scope}

\begin{scope}[shift = {(+2.0,0)}]
\ytableausetup{notabloids}
\ytableausetup{mathmode, boxsize=2.0em}
\node (n) {\ytableausetup{nosmalltableaux}
\ytableausetup{notabloids}
\ydiagram[*(white)]{2,2}};

\def \a {0.84};
\def \r {\a/4};

\shade[shading=ball, ball color = gray] ({-3.5*\a +3.5},{0.0*\a-0.425}) circle (\r);
\end{scope}

\end{tikzpicture}
\caption{A local excitation move.} \label{fig:lem}
\end{figure}
\end{center}

An {\em excited diagram} of $\mu$ in $\lambda$ is defined to be the cells of the Young diagram of $\lambda$ that are occupied by the pebbles after a sequence of local excitation moves. The collection of all excited diagrams of $\mu$ in $\lambda$ is denoted $\cE_\mu(\lambda)$.

For each cell $(i,j)$ inside the Young diagram of $\lambda$, let $h_{i,j}$ denote the hook length of the cell $(i,j)$ inside $\lambda$. We define the {\em excitation factor} of $\mu$ in $\lambda$, denoted $E_\mu (\lambda)$, as
\[
E_\mu (\lambda) := \sum_{\varepsilon \in \cE_\mu(\lambda)} \prod_{(i,j) \in \varepsilon} h_{i,j}.
\]
In other words, $E_\mu (\lambda)$ is the sum of the products of hook lengths of cells in the excited diagrams of $\mu$ in $\lambda$. When $\mu = [k]$ is a partition of one part, we also write $E_k(\lambda)$ instead of $E_{[k]}(\lambda)$ for convenience.

\begin{ex}
Suppose that $\lambda = [4,3,3]$ and $\mu = [2]$. There are three excited diagrams of $\mu$ in $\lambda$, labeled as $\varepsilon_1$, $\varepsilon_2$, and $\varepsilon_3$ in Figure \ref{fig:exciteddiagrams2in433}. The products of hook lengths of cells in $\varepsilon_1$, $\varepsilon_2$, and $\varepsilon_3$ are $6 \times 5 = 30$, $6 \times 2 = 12$, and $3 \times 2 = 6$, respectively. Therefore,
\[
E_{2}([4,3,3]) = 30 + 12 + 6 = 48.
\]
\end{ex}

\begin{remark}
By convention, for every partition $\lambda$, we have $E_0(\lambda) = 1$. This is because $\cE_{[0]}(\lambda)$ has exactly one element, which is the empty excited diagram $\varnothing$ inside $\lambda$. The product of hook lengths in $\varnothing$ is an empty product, which is $1$, and therefore, $E_0(\lambda) = 1$.
\end{remark}

\begin{center}
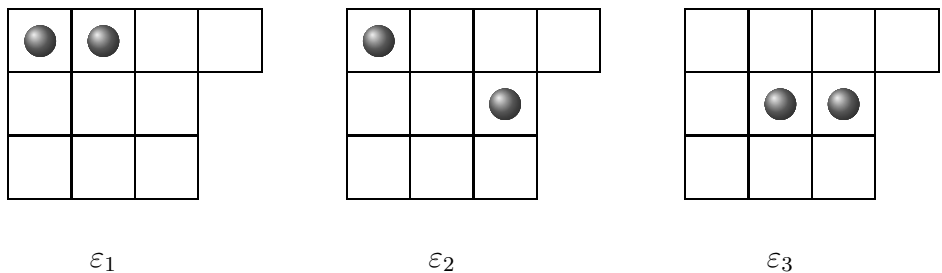
\begin{figure}
\begin{tikzpicture}
\begin{scope}[shift = {(-4.5,0)}]
\ytableausetup{notabloids}
\ytableausetup{mathmode, boxsize=2.0em}
\node (n) {\ytableausetup{nosmalltableaux}
\ytableausetup{notabloids}
\ydiagram[*(white)]{4,3,3}};

\def \a {0.84};
\def \r {\a/4};

\shade[shading=ball, ball color = gray] ({-5.5*\a +3.5},{1.5*\a-0.425}) circle (\r);
\shade[shading=ball, ball color = gray] ({-4.5*\a +3.5},{1.5*\a-0.425}) circle (\r);

\node at ({-4.5*\a +3.5},{-2.0*\a-0.425}) {$\varepsilon_1$};
\end{scope}

\begin{scope}[shift = {(0,0)}]
\ytableausetup{notabloids}
\ytableausetup{mathmode, boxsize=2.0em}
\node (n) {\ytableausetup{nosmalltableaux}
\ytableausetup{notabloids}
\ydiagram[*(white)]{4,3,3}};

\def \a {0.84};
\def \r {\a/4};

\shade[shading=ball, ball color = gray] ({-5.5*\a +3.5},{1.5*\a-0.425}) circle (\r);
\shade[shading=ball, ball color = gray] ({-3.5*\a +3.5},{0.5*\a-0.425}) circle (\r);

\node at ({-4.5*\a +3.5},{-2.0*\a-0.425}) {$\varepsilon_2$};
\end{scope}

\begin{scope}[shift = {(+4.5,0)}]
\ytableausetup{notabloids}
\ytableausetup{mathmode, boxsize=2.0em}
\node (n) {\ytableausetup{nosmalltableaux}
\ytableausetup{notabloids}
\ydiagram[*(white)]{4,3,3}};

\def \a {0.84};
\def \r {\a/4};

\shade[shading=ball, ball color = gray] ({-4.5*\a +3.5},{0.5*\a-0.425}) circle (\r);
\shade[shading=ball, ball color = gray] ({-3.5*\a +3.5},{0.5*\a-0.425}) circle (\r);

\node at ({-4.5*\a +3.5},{-2.0*\a-0.425}) {$\varepsilon_3$};
\end{scope}

\end{tikzpicture}
\caption{The three excited diagrams of $[2]$ in $[4,3,3]$} \label{fig:exciteddiagrams2in433}
\end{figure}
\end{center}

Having defined the excitation factors, we give Naruse's Skew-Shape Hook Length Formula as follows.

\begin{thm} {\em (Skew-Shape Hook Length Formula; Naruse \cite{Na14}, 2014)} 
\label{thm:NaruseSSHLF}
Let $m$ and $n$ be non-negative integers with $m \le n$. Let $\mu \vdash m$ and $\lambda \vdash n$ be partitions such that $\mu \le \lambda$. Then, the number of Standard Young Tableaux of skew shape $\lambda/\mu$ is given by
\[
N(\lambda/ \mu) = \frac{(n-m)!}{n!} \cdot N(\lambda) \cdot E_{\mu} (\lambda).
\]
\end{thm}

For the case in which $\mu = [m]$ is a partition with one part, the next corollary follows immediately from Naruse's formula.

\begin{cor}
Let $m$ and $n$ be non-negative integers with $m \le n$ and let $\lambda$ be a partition of $n$. Then,
\[
N(\lambda / [m]) = \frac{(n-m)!}{n!} \cdot N(\lambda) \cdot \underbrace{\sum_{i_1 \le \dots \le i_m} h_{i_1, i_1} \cdot h_{i_2, i_2+1} \cdot ~\cdots ~\cdot h_{i_m, i_m + m-1}}_{= E_m(\lambda)}
\]
where $h_{i,j}$ is the hook length of the cell $(i,j)$ in $\lambda$, and the sum calculates over all $(i_1, \dots, i_m)$ for which $(i_t, i_t+t-1)$ is in $\lambda$ for all $t = 1,2, \dots, m$. If $m$ is greater than $\lambda_1$, then by convention, $N(\lambda/ [m]) = 0$ and $E_m(\lambda) = 0$. \label{cor:EmFormula}
\end{cor}

We note that even in the case as simple as $\mu = [m]$, the above formula for the excitation factor $E_m(\lambda)$ is complicated and depends heavily upon the shape of $\lambda$. In the next section, we will explore the excitation factor $E_m(\lambda)$ in more detail, and give an alternative expression for it.

\bigskip

\section{The SONAR Technique} \label{sec:SONAR}
We first consider the following question. Suppose that the partition $\lambda$ is initially unknown. If the excitation factors $E_m(\lambda)$ are known for all non-negative integers $m$, can we recover the partition $\lambda$? That is, can the sequence $\{E_m(\lambda)\}_{m=0}^{\infty}$ recover $\lambda$?

We refer to this method of recovery of the partition $\lambda$ as the \enquote{SONAR} technique. Suppose we imagine the partition $\lambda$ as an ocean whose exact shape we would like to determine. We can remove $m$ leftmost cells on the top row of $\lambda$ to obtain information of $E_m(\lambda)$ for all $m \ge 0$. Imagine this as a boat on the surface of the ocean $\lambda$ that can take out as many cells on the surface from the upper-left corner of $\lambda$ as it wishes, but cannot go deeper than that. (See Figure \ref{fig:SONAR}.) Can the boat retrieve the whole topography of the ocean $\lambda$? The name SONAR comes from the similarity of this technique to the actual SONAR technique in oceanography, by which a boat on the surface can retrieve the whole topography of the ocean without having to go deeper than the surface.

\begin{center}
\begin{figure}
\begin{tikzpicture}
\begin{scope}[shift = {(0,0)}]
\ytableausetup{notabloids}
\ytableausetup{mathmode, boxsize=2.0em}
\node (n) {\ytableausetup{nosmalltableaux}
\ytableausetup{notabloids}
\ydiagram[*(gray!20)]{4}
*[*(white)]{4+4,6,5,2,1,1}};

\def \a {0.84};
\def \r {\a/4};
\end{scope}

\def \a {0.84};

\begin{scope}[shift = {({-5*\a/6},{37*\a/12-32*\a/405})}]
\boat[]
\draw[->,thick] ({7*\a/6},{\a/3}) -- ({10*\a/6},{\a/3});
\end{scope}

\end{tikzpicture}
\caption{SONAR technique} \label{fig:SONAR}
\end{figure}
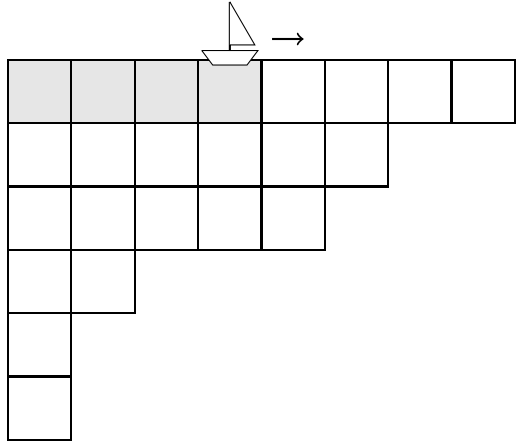
\end{center}

In this section, we prove that the SONAR technique always works. We state this as the following theorem.

\begin{restatable}[SONAR Technique]{thm}{SONAR} 
\label{thm:SONAR}
Let $\lambda$ be a partition. Then, the sequence $\{E_m(\lambda)\}_{m=0}^{\infty}$ uniquely determines $\lambda$. 
\end{restatable}

Our proof of Theorem \ref{thm:SONAR} uses the following theorem.

\begin{restatable}{thm}{Qklincomb}
\label{thm:Qklincomb}
Let $\lambda$ be a partition of a positive integer $n$. Suppose that the hook lengths in the first row of $\lambda$ are $a_1 > a_2 > \dots > a_k$ where $k = \lambda_1$. Then, for every non-negative integer $m$, we have
\[
E_m(\lambda) = \sum_{i=0}^m \binom{k-i}{m-i} S_{m-i}(m-k-1) \cdot \sigma_i (a_1, a_2, \dots, a_k),
\]
where $\sigma_i(a_1, a_2, \dots, a_k) = \sum_{1 \le \iota_1 < \iota_2 < \dots < \iota_i \le k} a_{\iota_1} a_{\iota_2} \cdots a_{\iota_i}$ is the $i$-th elementary symmetric polynomial in the variables $a_1, \dots, a_k$, with $\sigma_0 = 1$ by convention, and $S_j$ is the $j$-th Stirling polynomial.
\end{restatable}

Theorem \ref{thm:Qklincomb} expresses the excitation factor $E_m(\lambda)$ as an explicit $\mathbb{Q}[k]$-linear combination of $\sigma_i(a_1, a_2, \dots, a_k)$. The goal of this section is to prove Theorem \ref{thm:Qklincomb}, and then uses it to prove Theorem \ref{thm:SONAR}. The proof of Theorem \ref{thm:Qklincomb} requires algebraic identities of the Stirling polynomials $S_j$. We will spend the following part establishing these identities which we will need later.

\medskip

\subsection{The Stirling Polynomials} \label{ss:algpropci}
The Stirling polynomials are well-studied objects in Algebraic Combinatorics with many equivalent definitions. (See for example, \cite{EMOT81, Ro05, Stan13}.) For our purpose, we will use the following definition, which links the Stirling polynomials $S_j$ to the coefficients of the polynomial of the form $(X+1)(X+2) \cdots (X+\ell)$ for a positive integer $\ell$. For each non-negative integer $j \ge 0$, we define $S_j(x) \in \mathbb{Q}[x]$ to be the polynomial such that for every positive integer $\ell$, the equation
\[
\binom{\ell}{j} S_j(\ell) = \sigma_j (\ell, \dots, 1)
\]
holds. Let $c_j(x)$ denote the polynomial $\binom{x}{j} S_j(x) \in \mathbb{Q}[x]$. With this notation, the formula in Theorem \ref{thm:Qklincomb} can be written as
\[
E_m(\lambda) = \sum_{i=0}^m (-1)^{m-i} c_{m-i}(m-k-1) \cdot \sigma_i(a_1, a_2, \dots, a_k).
\]
The first few Stirling polynomials can be calculated explicitly as follows:
\begin{itemize}
\item $S_0(x) = 1$,
\item $S_1(x) = \frac{1}{2}(x+1)$,
\item $S_2(x) = \frac{1}{12}(x+1)(3x+2)$,
\item $S_3(x) = \frac{1}{8}x(x+1)^2$, and
\item $S_4(x) = \frac{1}{240} (x+1)(15x^3+15x^2-10x-8)$.
\end{itemize}

In the proof of Theorem \ref{thm:Qklincomb}, we will work with the elementary symmetric polynomials of lists of real numbers. In particular, we will be interested in how the value of $\sigma_k$, for each $k$, changes when we transform the list in certain ways. In the following lemma, we will see that when we add the positive integers $\ell, \ell-1, \dots, 1$ to a list of real numbers, the value of $\sigma_k$ of the new list can be described as a linear combination of the values of $\sigma_j$ of the old list with coefficients in the form of the Stirling polynomials.

\begin{lemma}
Let $k, \ell, N \in \mathbb{Z}_{\ge 0}$, and $x_1, \dots, x_N \in \mathbb{R}$. Then,
\[
\sigma_k(x_1, x_2, \dots, x_N, \ell, \ell -1, \dots, 1) = \sum_{j=0}^k \binom{\ell}{k-j} S_{k-j}(\ell) \cdot \sigma_j(x_1, \dots, x_N).
\] \label{l:sigma_concatenate_l}
\end{lemma}

\begin{proof}
The quantity $\sigma_k(x_1, \dots, x_N, \ell, \dots, 1)$ is the sum of the products of $k$ elements from $x_1, \dots, x_N, \ell, \dots, 1$. We can write this sum as
\[
\sigma_k(x_1, \dots, x_N, \ell, \dots, 1) = A_0 + A_1 + \dots + A_k
\]
where $A_j$ is the sum of the products of $k$ elements from $x_1, \dots, x_N, \ell, \dots, 1$ with exactly $j$ multiplicands from $x_1, \dots, x_N$ and exactly $k-j$ multiplicands from $\ell, \dots, 1$. Thus,
\[
A_j = \sigma_j(x_1, \dots, x_N) \cdot \sigma_{k-j} (\ell, \dots, 1) = \binom{\ell}{k-j} S_{k-j}(\ell) \cdot \sigma_j(x_1, \dots, x_N).
\]
Summing $A_j$ for $j=0,1,\dots, k$ yields the desired result.
\end{proof}

If we add a real number to every entry of a list, the value $\sigma_k$ is changed as follows.

\begin{lemma}
Let $k, N \in \mathbb{Z}_{\ge 0}$, and $x_1, \dots, x_N, d \in \mathbb{R}$. Then,
\[
\sigma_k(x_1 + d, x_2 + d, \dots, x_N + d) = \sum_{j=0}^k \binom{N-j}{k-j} \cdot d^{k-j} \cdot \sigma_j(x_1, \dots, x_N).
\] \label{l:sigma_plus_d}
\end{lemma}

\begin{proof}
Recall that for $P(X) \in \mathbb{C}[X]$, we use the notation $[P(X)]_{X^i}$ to denote the coefficient of $X^i$ of $P(X)$. We have
\begin{align*}
\sigma_k(x_1+d, \dots, x_N+d) &= \left[ (X+(x_1+d)) \cdots (X+(x_N+d)) \right]_{X^{N-k}} \\
&= \left[ ((X+d)+x_1) \cdots ((X+d)+x_N) \right]_{X^{N-k}} \\
&= \left[ \sum_{j=0}^N (X+d)^{N-j} \sigma_j(x_1, \dots, x_N) \right]_{X^{N-k}} \\
&= \sum_{j=0}^k \binom{N-j}{k-j} d^{k-j} \sigma_j(x_1, \dots, x_N).
\end{align*}
Note that in the final equality, we changed the upper limit of the sum from $N$ to $k$. This is valid because, for $j>k$, there does not exist an $X^{N-k}$-term in $(X+d)^{N-j}$.
\end{proof}

Next, we will prove a general lemma about polynomials.

\begin{lemma}
Let $B_1, B_2, B_3$ be integers. Suppose that the bivariate polynomial $P(X,Y) \in \mathbb{C}[X,Y]$ satisfies $P(x,y) = 0$ for all $(x,y) \in \mathbb{Z} \times \mathbb{Z}$ such that $x \le B_1$, $y \ge B_2$, and $x+y \le B_3$. Then, $P$ is the zero polynomial in $\mathbb{C}[X,Y]$. \label{l:polyptscheck}
\end{lemma}

\begin{proof}
If $P$ is not identically zero, write the polynomial $P(X,Y)$ as
\[
P(X,Y) = p_k(Y) \cdot X^k + p_{k-1}(Y) \cdot X^{k-1} + \dots + p_0(Y)
\]
where $p_i(Y) \in \mathbb{C}[Y]$ and $k$ is the $X$-degree of $P$.

For each fixed $y_0 \in \mathbb{Z}_{\ge B_2}$, we have that $P(x,y_0) = 0$ for all integers $x \le \min\{B_1, B_3 - y_0\}$. Therefore, $P(X,y_0) \in \mathbb{C}[X]$ has infinitely many zeros, and so $P(X,y_0)$ must be identically zero in $\mathbb{C}[X]$. This shows that $p_k(y_0) = p_{k-1}(y_0) = \dots = p_0(y_0) = 0$ for each $y_0 \in \mathbb{Z}_{\ge B_2}$. Now, by varying $y_0$, each $p_i$ also has infinitely many zeros, and therefore, $p_i$ is identically zero for each $i= 0,1, \dots, k$. Thus, $P$ must be the zero polynomial.
\end{proof}

In the following, we present an important lemma about the Stirling polynomials.

\begin{lemma}
Let $M \in \mathbb{Z}_{\ge 0}$ and $q,r \in \mathbb{Z}$. Then,
\begin{align*}
&\sum_{i=0}^M \binom{q+r-i-1}{M-i} \binom{q-1}{i} S_{M-i}(M-q-r) S_i(q-1) \\
&= \sum_{i=0}^M (-1)^i \binom{r-i-1}{M-i} \binom{r-1}{i} S_{M-i}(M-r) q^i
\end{align*} \label{l:maincidentity}
\end{lemma}

\begin{proof}
For each fixed $M \in \mathbb{Z}_{\ge 0}$, we will prove the identity for $(q,r) \in \mathbb{Z} \times \mathbb{Z}$ such that $r \le 0$, $q \ge 2$, and $r+q \le M-1$. The result will be sufficient for us to conclude that the identity is true for all $(q,r) \in \mathbb{Z} \times \mathbb{Z}$ by Lemma \ref{l:polyptscheck}.

For these choices of $(q,r)$, we have $M-q-r \ge 1$, $q-1 \ge 1$, and $M-r > 0$. We note that the left-hand-side of the above equation is
\begin{align*}
&\sum_{i=0}^M \binom{q+r-i-1}{M-i} \binom{q-1}{i} S_{M-i}(M-q-r) S_i(q-1) \\
&= \sum_{i=0}^M (-1)^{M-i} \binom{M-q-r}{M-i} S_{M-i}(M-q-r) \cdot \binom{q-1}{i} S_i(q-1) \\
&= \sum_{i=0}^M (-1)^{M-i} \cdot \sigma_{M-i}(1,2, \dots, M-q-r) \cdot \sigma_i(1, \dots, q-1) \\
&= \sum_{i=0}^M \sigma_{M-i} (-(M-q-r), -(M-q-r)+1, \dots, -1) \cdot \sigma_i(1,\dots, q-1) \\
&= \sigma_M (-(M-q-r), -(M-q-r)+1, \dots, -1,0,1, \dots, q-1).
\end{align*}
The last equality is due to the double-counting argument which computes
\[
\sum_{-(M-q-r) \le \alpha_1 < \dots < \alpha_M \le q-1} \alpha_1 \cdot ~\cdots ~\cdot \alpha_M
\]
in two ways. The first way is to consider how many of the $\alpha_j$'s are positive, and add up the sums from the different cases, while the second way is directly by the definition of $\sigma_M$.

Next, we apply Lemma \ref{l:sigma_plus_d} to the quantity above to obtain the following.
\begin{align*}
&\sigma_M (-(M-q-r), -(M-q-r)+1, \dots, -1,0,1, \dots, q-1) \\
&= (-1)^M \cdot \sigma_M \left( 1-q, 2-q, \dots, (M-r)-q \right) \\
&\overset{\text{(Lemma \ref{l:sigma_plus_d})}}= (-1)^M \cdot \sum_{i=0}^M \binom{M-r-i}{M-i} (-q)^{M-i} \cdot \sigma_i(1, \dots, M-r) \\
&= (-1)^M \cdot \sum_{i=0}^M \binom{i-r}{i} (-q)^i \sigma_{M-i}(1, \dots, M-r) \\
&= \sum_{i=0}^M (-1)^i \binom{r-i-1}{M-i} \binom{r-1}{i} S_{M-i}(M-r) q^i.
\end{align*}
The second-to-last equality above is obtained by reindexing $i \leftrightarrow M-i$ in the sum. The last equality is due to the fact that
\[
(-1)^i \binom{i-r}{i} = \binom{r-1}{i} \quad \text{and} \quad (-1)^{M-i} \binom{M-r}{M-i} = \binom{r-i-1}{M-i}
\]
for $0 \le i \le M$. These equations hold even when $r$ is not positive. We have finished the proof of the lemma.
\end{proof}

A special case of Lemma \ref{l:maincidentity} gives a particularly nice result.

\begin{cor}
Let $m \in \mathbb{Z}_{\ge 0}$ and $k \in \mathbb{Z}$. Then,
\[
\sum_{i=0}^m \binom{m}{i} S_{m-i} (m-k-1) S_i(k) = m!
\] \label{c:EmRow}
\end{cor}
\begin{proof}
By substituting $M \mapsto m$, $q \mapsto k+1$, and $r \mapsto 0$ in Lemma \ref{l:maincidentity}, we obtain 
\footnotesize
\[ 
\sum_{i=0}^m \binom{k-i}{m-i} \binom{k}{i} S_{m-i}(m-k-1) S_i(k) = \sum_{i=0}^m (-1)^i \binom{-i-1}{m-i} \binom{-1}{i} S_{m-i}(m) (k+1)^i. 
\]
\normalsize Therefore,
\footnotesize
\[
\binom{k}{m} \cdot \sum_{i=0}^m \binom{m}{i} S_{m-i} (m-k-1) S_i(k) = \sum_{i=0}^m (-1)^{m-i} \sigma_{m-i} (m,m-1, \dots, 1) \cdot (k+1)^i.
\] \normalsize
Note that the right hand side is precisely the expansion of 
\[
((k+1)-1)((k+1)-2) \cdots ((k+1)-m) = \binom{k}{m} \cdot m!.
\]
Hence,
\[
\binom{k}{m} \cdot \sum_{i=0}^m \binom{m}{i} S_{m-i}(m-k-1) S_i(k) = \binom{k}{m} \cdot m!.
\]
The equation above is true for every integer $k$. Therefore, for a fixed $m$, the equation is an equality of polynomials in $k$, and so we may cancel $\binom{k}{m}$ from both sides to obtain
\[
\sum_{i=0}^m \binom{m}{i} S_{m-i}(m-k-1) S_i(k) = m!
\]
as desired.
\end{proof}

\medskip

\subsection{Recovering Partitions}
After having established a few algebraic identities, we will prove that the SONAR technique can always recover the partition. In the following, we will prove Theorem \ref{thm:Qklincomb}, which we restate here.

\Qklincomb*

\begin{center}
\begin{figure}
\begin{tikzpicture}
\footnotesize
\def \a {0.84};
\def \r {\a/4};

\draw (0,7*\a) -- (10*\a,7*\a);
\draw (0,6*\a) -- (10*\a,6*\a);
\draw (2*\a,2*\a) -- (5*\a,2*\a);
\draw (0,0*\a) -- (2*\a,0*\a);
\draw (2*\a,2*\a) -- (5*\a,2*\a);
\draw (0*\a,7*\a) -- (0*\a,0*\a);
\draw (1*\a,7*\a) -- (1*\a,0*\a);
\draw (2*\a,2*\a) -- (2*\a,0*\a);
\draw (5*\a,7*\a) -- (5*\a,2*\a);
\draw (2*\a,7*\a) -- (2*\a,6*\a);
\draw (4*\a,7*\a) -- (4*\a,6*\a);
\draw (6*\a,7*\a) -- (6*\a,6*\a);
\draw (7*\a,7*\a) -- (7*\a,6*\a);
\draw (9*\a,7*\a) -- (9*\a,6*\a);
\draw (10*\a,7*\a) -- (10*\a,6*\a);

\node at (3*\a,6.5*\a) {\normalsize $\dots$};
\node at (8*\a,6.5*\a) {\normalsize $\dots$};
\node at (3*\a,4*\a) {\normalsize $\ddots$};
\node at (0.5*\a,4*\a) {\normalsize $\vdots$};

\node at (0.5*\a,6.5*\a) {$a_1$};
\node at (1.5*\a,6.5*\a) {$a_2$};
\node at (4.5*\a,6.5*\a) {$a_p$};
\node at (5.5*\a,6.5*\a) {$\ell$};
\node at (6.5*\a,6.5*\a) {$\ell-1$};
\node at (9.5*\a,6.5*\a) {$1$};

\end{tikzpicture}
\caption{Diagram for the hook lengths of $\lambda$ in the proof of Theorem \ref{thm:Qklincomb}} \label{fig:ProofThmSInd}
\end{figure}
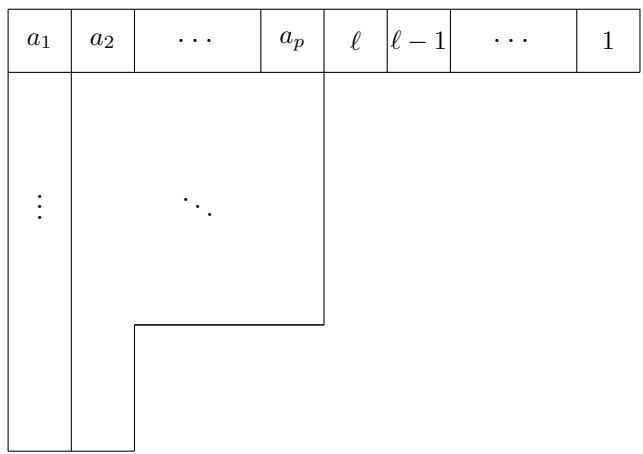
\end{center}

\begin{proof}[Proof of Theorem \ref{thm:Qklincomb}]
We proceed by strong induction on $k$, which is the number of cells in the first row of $\lambda$. For the base step when $k=1$, the partition $\lambda$ is a column partition 
\[
\lambda = [\underbrace{1,1, \dots, 1}_{a_1 \text{ copies of } 1 \text{'s}}].
\]
In this case, for $m=0$, $E_0(\lambda) = 1 = \binom{1}{0} S_0(-2) \sigma_0(a_1)$. For $m=1$, we have $E_1(\lambda) = a_1$ and the right-hand-side sum is 
\[
\binom{1}{1} S_1(-1) \sigma_0(a_1) + \binom{0}{0} S_0(-1) \sigma_1(a_1) = a_1.
\]
For $m \ge 2$, which is greater than $k$, we expect $E_m(\lambda)$ to be zero. We will show that the formula works in this case as well. Since $k=1$, the term $\sigma_i(a_1, \dots, a_k) = \sigma_i(a_1)$ vanishes for all $i \ge 2$. Thus, the right hand side equals
\[
\binom{1}{m} S_m(m-2) \sigma_0(a_1) + \binom{0}{m-1} S_{m-1}(m-2) \sigma_1(a_1),
\]
which is zero as we expected. Therefore, the claim is true for the base step.

For the inductive step, suppose that $\lambda_1 = p+\ell$ and $\lambda_2 = p$ where $p$ and $\ell$ are non-negative integers. (See Figure \ref{fig:ProofThmSInd}.) In this case, $k = p+\ell$. First, we consider the case in which $p=0$, which is when $\lambda$ is a partition of one part, $\lambda = [\ell] = [k]$. In this case, for each integer $m \ge 0$, we obtain 
\[
E_m(\lambda) = k(k-1) \cdots (k-m+1) = \binom{k}{m} \cdot m!.
\]
The right-hand-side expression is
\footnotesize
\[
\sum_{i=0}^m \binom{k-i}{m-i} S_{m-i} (m-k-1) \sigma_i(k,k-1, \dots, 1) = \binom{k}{m} \cdot \sum_{i=0}^m \binom{m}{i} S_{m-i}(m-k-1) S_i(k)
\] \normalsize
which is precisely $\binom{k}{m} \cdot m!$ by Corollary \ref{c:EmRow}.

From now on, suppose $p \ge 1$. Let $\wt{\lambda}$ denote the partition obtained by removing the first (leftmost) column from $\lambda$ and let $\wh{\lambda}$ denote the partition obtained by removing the first (uppermost) row from $\wt{\lambda}$. We note that the hook lengths of the first row of $\wt{\lambda}$ are $a_2, a_3, \dots, a_p, \ell, \ell - 1, \dots, 1$. The hook lengths of the first row of $\wh{\lambda}$ are $a_2-(\ell+1), a_3-(\ell+1), \dots, a_p - (\ell + 1)$. (The latter list is empty when $p=1$.)

By Corollary \ref{cor:EmFormula}, the quantity $E_m(\lambda)$ is the sum of the products of $m$ hook lengths of the form
\[
h_{i_1,i_1} \cdot h_{i_2, i_2+1} \cdot ~\cdots ~\cdot h_{i_m, i_m+(m-1)}
\]
where each $h_{i,j}$ is the hook length in $\lambda$ of the $(i,j)$-cell. We can split the summands of this sum into two groups: the first with $h_{i_1, i_1} = h_{1,1} = a_1$, and the second with $h_{i_1, i_1} \neq h_{1,1}$. In the first case, the rest of the product $h_{i_2, i_2+1} \cdot ~\cdots ~\cdot h_{i_m, i_m+(m-1)}$ is actually the product of hook lengths in an excited diagram of $[m-1]$ inside $\wt{\lambda}$. In the second case, the product itself is the product of hook lengths in an excited diagram of $[m]$ inside $\wh{\lambda}$. This gives the following recursion:
\[
E_m(\lambda) = a_1 \cdot E_{m-1}(\wt{\lambda}) + E_m(\wh{\lambda}).
\]
Note that the numbers of cells in the first rows of $\wt{\lambda}$ and $\wh{\lambda}$ are $k-1$ and $k-\ell-1$, respectively, which are both less than $k$. Thus, we may use the induction hypothesis for $E_{m-1}(\wt{\lambda})$ and $E_m(\wh{\lambda})$.

Hence,
\footnotesize
\begin{align*}
&E_m(\lambda) = a_1 \cdot \sum_{i=0}^{m-1} \binom{p+\ell-1-i}{m-1-i} S_{m-1-i}(m-1-(p+\ell-1)-1) \sigma_i(a_2, \dots, a_p, \ell, \dots, 1) \\
&\hphantom{E_m(\lambda) = } + \sum_{i=0}^m \binom{p-1-i}{m-i} S_{m-i}(m-(p-1)-1) \sigma_i(a_2-(\ell+1), \dots, a_p-(\ell+1)) \\
&= \sum_{i=0}^{m-1} \binom{p+\ell-1-i}{m-1-i} S_{m-1-i}(m-p-\ell-1) \\
&\hphantom{= \sum_{i=0}^{m-1} \binom{p+\ell-1-i}{m-1-i}} \cdot \left(\sigma_{i+1}(a_1, \dots, a_p,\ell, \dots, 1) - \sigma_{i+1}(a_2, \dots, a_p, \ell, \dots, 1) \right) \\
&\hphantom{=} + \sum_{i=0}^m \binom{p-1-i}{m-i} S_{m-i}(m-p) \sigma_i (a_2-(\ell+1), \dots, a_p - (\ell+1)) \\
&= \sum_{i=0}^m \binom{p+\ell-i}{m-i} S_{m-i}(m-p-\ell-1) \cdot \left( \sigma_i(a_1, \dots, a_p,\ell,\dots, 1) - \sigma_i(a_2, \dots, a_p, \ell, \dots, 1) \right) \\
&\hphantom{=} + \sum_{i=0}^m \binom{p-1-i}{m-i} S_{m-i}(m-p) \cdot \sigma_i(a_2-(\ell+1), \dots, a_p-(\ell+1)).
\end{align*}
\normalsize
The last equality is obtained simply by an index shift of $i$. We now have
\begin{align*}
E_m(\lambda) &= \left[ \sum_{i=0}^m \binom{p+\ell-i}{m-i} S_{m-i}(m-(p+\ell)-1) \cdot \sigma_i(a_1, \dots, a_p, \ell, \dots, 1) \right] \\
& \hphantom{=} - \left[ \sum_{i=0}^m \binom{p+\ell-i}{m-i} S_{m-i}(m-p-\ell-1) \cdot \sigma_i(a_2, \dots, a_p, \ell, \dots, 1) \right] \\
& \hphantom{=} + \left[ \sum_{i=0}^m \binom{p-1-i}{m-i} S_{m-i}(m-p) \cdot \sigma_i(a_2-(\ell+1), \dots, a_p-(\ell+1)) \right].
\end{align*}
Note that the first term on the right hand side is precisely what we want to equal to $E_m(\lambda)$. Therefore, it suffices to prove the equality
\begin{align*} \tag{$\triangle$}
&\left[ \sum_{i=0}^m \binom{p+\ell-i}{m-i} S_{m-i}(m-p-\ell-1) \cdot \sigma_i(a_2, \dots, a_p, \ell, \dots, 1) \right] \\
&= \left[ \sum_{i=0}^m \binom{p-1-i}{m-i} S_{m-i}(m-p) \cdot \sigma_i(a_2-(\ell+1), \dots, a_p-(\ell+1)) \right]. 
\end{align*}

Fortunately, we have done the difficult algebraic work in Section \ref{ss:algpropci}. To show the equality ($\triangle$), we only need to apply lemmas we proved earlier to transform the left hand side $\LHS_{(\triangle)}$ to the right hand side $\RHS_{(\triangle)}$. To start with, Lemma \ref{l:sigma_concatenate_l} gives $\LHS_{(\triangle)} = $
\[
\sum_{j=0}^m \left[ \sigma_j(a_2, \dots, a_p) \cdot \left( \sum_{i=j}^m \binom{p+\ell-i}{m-i} \binom{\ell}{i-j} S_{m-i}(m-p-\ell-1) S_{i-j}(\ell) \right)\right].
\]
Lemma \ref{l:sigma_plus_d} gives $\RHS_{(\triangle)} = $
\footnotesize
\[
\sum_{j=0}^m \left[ \sigma_j(a_2, \dots, a_p) \cdot \left( \sum_{i=j}^m (-1)^{i-j} \binom{p-1-i}{m-i} \binom{p-1-j}{i-j} S_{m-i}(m-p) \cdot (\ell+1)^{i-j} \right) \right].
\]
\normalsize
It suffices to show that, for each $0 \le j \le m$,
\begin{align*}
& \sum_{i=j}^m \binom{p+\ell-i}{m-i} \binom{\ell}{i-j} S_{m-i}(m-p-\ell-1) S_{i-j}(\ell) \\
& = \sum_{i=j}^m (-1)^{i-j}\binom{p-1-i}{m-i} \binom{p-1-j}{i-j} S_{m-i}(m-p) \cdot (\ell+1)^{i-j}.
\end{align*}
By shifting indices ($i \mapsto i+j$), the equation above is equivalent to
\begin{align*}
& \sum_{i=0}^{m-j} \binom{p+\ell-i-j}{m-i-j} \binom{\ell}{i} S_{m-i-j}(m-p-\ell-1)S_i(\ell) \\
& = \sum_{i=0}^{m-j} (-1)^i \binom{p-1-i-j}{m-i-j} \binom{p-1-j}{i} S_{m-i-j}(m-p) (\ell+1)^i.
\end{align*}
We can see that the last equation is true by substituting $M \mapsto m-j \ge 0$, $q \mapsto \ell+1$, and $r \mapsto p-j$ in Lemma \ref{l:maincidentity}. We have finished the induction and hence completed the proof of the theorem.
\end{proof}

\begin{remark} \label{rmk:PixtonsFormulas}
Aaron Pixton \cite{Pix16} observed that the formula in Theorem \ref{thm:Qklincomb} can be rewritten as
\[
E_m(\lambda) = \left[ \frac{(1+a_1 X) \cdots (1+a_k X)}{(1+X)(1+2X) \cdots (1+(k-m)X)}\right]_{X^m}.
\]
Also,
\[
E_{\big[\underbrace{1,1,\dots,1}_{m}\big]}(\lambda) = \left[ \frac{(1-X)(1-2X) \cdots (1-(k+m-2)X)}{(1-a_1X) \cdots (1-a_kX)}\right]_{X^m}.
\]
\end{remark}

Next, we show that the SONAR technique always works. This is Theorem \ref{thm:SONAR}, which we reproduce here.

\SONAR*

\begin{proof}[Proof of Theorem \ref{thm:SONAR}]
Given the sequence $\{E_m(\lambda)\}_{m=0}^{\infty}$, we can recover $k = \lambda_1$, the number of cells in the first row easily by searching for the first zero excitation factor $E_m(\lambda)$. Indeed, $E_m(\lambda)$ is positive if and only if $m \le k$. Because the $i$-th entry of the sequence is $E_{i-1}(\lambda)$, if the first zero in the sequence appears in the $i$-th entry where $i \ge 3$, then $k = i-2$. If the first zero appears in the second entry (that is, $E_1(\lambda) = 0$), then $\lambda$ is the empty partition. By convention, the first entry $E_0(\lambda) = 1$ is never zero. 

Now that we have recovered $k$, we continue to recover the hook lengths $a_1 > \dots > a_k$ of the first row of $\lambda$. The formula in Theorem \ref{thm:Qklincomb} can be rewritten as
\[
E_m(\lambda) = \sum_{i=0}^m \gamma(m,i) \cdot \sigma_i(a_1, \dots, a_k)
\]
where $\gamma(m,i) = \binom{k-i}{m-i} S_{m-i}(m-k-1)$. For convenience, we will use $\sigma_i$ to refer to $\sigma_i(a_1, \dots, a_k)$. Consider the $(k+1) \times (k+1)$-matrix $\cA_{k+1} := [\alpha_{i,j}]_{1 \le i, j \le k+1}$, where
\[
\alpha_{i,j} = \begin{cases}
\gamma(i-1,j-1) & \text{if } i \ge j \\
0 & \text{otherwise.}
\end{cases}
\]
The formula above shows that
\[
\cA_{k+1} \cdot \begin{bmatrix}
\sigma_0 \\
\sigma_1 \\
\vdots \\
\sigma_k
\end{bmatrix}
= \begin{bmatrix}
E_0(\lambda) \\
E_1(\lambda) \\
\vdots \\
E_k(\lambda)
\end{bmatrix}.
\]
Note that all the entries in the matrix $\cA_{k+1}$ are known, as they depend only on $k$. Furthermore, it is lower triangular with all the diagonal entries being $1$. Therefore, $\cA_{k+1}$ is invertible. As we know $\{E_m(\lambda)\}_{m=0}^\infty$, we can recover $\sigma_0, \dots, \sigma_k$ from the equation
\[
\begin{bmatrix}
\sigma_0 \\
\sigma_1 \\
\vdots \\
\sigma_k
\end{bmatrix}
= \cA_{k+1}^{-1} \cdot \begin{bmatrix}
E_0(\lambda) \\
E_1(\lambda) \\
\vdots \\
E_k(\lambda)
\end{bmatrix}.
\]
Consequently, we recover the polynomial
\[
A(X) := (X-a_1)(X-a_2) \cdots (X-a_k) = \sum_{i=0}^k (-1)^i \sigma_i \cdot X^{k-i}.
\]
Since $a_1, a_2, \dots, a_k$ are strictly decreasing, the roots of $A(X)$ are pairwise distinct positive integers. Therefore, solving $A(X) = 0$ recovers all the hook lengths in the first row of $\lambda$, which in turn recover the whole partition.
\end{proof}

To demonstrate the SONAR technique in action, consider the following example.

\begin{ex} \label{ex:1,8,28,40}
Suppose that $\lambda$ is a partition which satisfies 
\[
\{E_m(\lambda)\}_{m=0}^{\infty} = (1,8,28,40,0,0,\dots).
\]
We will recover $\lambda$ from the sequence above using SONAR.

First, because there are four positive terms in the sequence, we have $k=3$. We calculate the matrix $\cA_4$ as follows.
\[
\cA_4 = \begin{bmatrix}
\gamma(0,0) & 0 & 0 & 0 \\
\gamma(1,0) & \gamma(1,1) & 0 & 0 \\
\gamma(2,0) & \gamma(2,1) & \gamma(2,2) & 0 \\
\gamma(3,0) & \gamma(3,1) & \gamma(3,2) & \gamma(3,3)
\end{bmatrix} = \begin{bmatrix*}[r]
1 & 0 & 0 & 0 \\
-3 & 1 & 0 & 0 \\
1 & -1 & 1 & 0 \\
0 & 0 & 0 & 1
\end{bmatrix*}.
\]
Therefore,
\[
\cA_4^{-1} = \begin{bmatrix}
1 & 0 & 0 & 0 \\
3 & 1 & 0 & 0 \\
2 & 1 & 1 & 0 \\
0 & 0 & 0 & 1
\end{bmatrix}.
\]
We can then recover $\sigma_i$ for $i = 0,1,2,3$:
\[
\begin{bmatrix}
\sigma_0 \\
\sigma_1 \\
\sigma_2 \\
\sigma_3
\end{bmatrix}
= \begin{bmatrix}
1 & 0 & 0 & 0 \\
3 & 1 & 0 & 0 \\
2 & 1 & 1 & 0 \\
0 & 0 & 0 & 1
\end{bmatrix} \cdot \begin{bmatrix}
1 \\
8 \\
28 \\
40
\end{bmatrix} = \begin{bmatrix}
1 \\
11 \\
38 \\
40
\end{bmatrix}.
\]
Thus, $A(X) = X^3 - 11X^2 + 38X - 40 = (X-5)(X-4)(X-2)$. We conclude that $a_1 = 5$, $a_2 = 4$, and $a_3 = 2$. Therefore, $\lambda = [3,3,2]$.
\end{ex}

Now that we have developed the SONAR technique, in the next section we will use it to explore the asymptotic behavior of the function $G(n)$ as $n \rightarrow \infty$.

\bigskip

\section{Asymptotic Behavior of $G(n)$} \label{sec:asympGn}
In this section, we will show that $n-G(n) = O(n/\log n)$ as $n \rightarrow \infty$. Our main tool in the proof is the SONAR technique we developed in Section \ref{sec:SONAR}.

\begin{restatable}{thm}{Gnsimn} 
\label{thm:Gnsimn}
We have
\[
\lim_{n \rightarrow \infty} \frac{G(n)}{n} = 1
\]
and $n-G(n) = O(n/\log n)$ as $n \rightarrow \infty$.
\end{restatable}

Before proving Theorem \ref{thm:Gnsimn}, we show the following lemma.

\begin{lemma}
\label{l:rectangle}
Let $\lambda$ be a partition of a positive integer $n \ge 2$. Then, there exists a partition $\mu = [\underbrace{a,a,\dots,a}_{b}]$ of a positive integer $ab$ such that $\mu \nleqslant \lambda$ but $\wt{\mu} \le \lambda$, where $\wt{\mu}$ is the unique $1$-minor of $\mu$, and such that
\[
|\mu| = ab < \frac{2n}{\log n}.
\]
\end{lemma}

\begin{proof}
Let $m := \lfloor 2n/\log n \rfloor$. If we show that there exists a partition $\tau$ of size at most $m$ such that $\tau \nleqslant \lambda$, then we can find a cell $X$ inside $\tau$ that is not in $\lambda$ such that all the cells to its left and all the cells above it are in $\lambda$. Then, we may choose $\mu$ to be the rectangle partition that has $X$ as its corner square. We observe that $\mu$ is not in $\lambda$, but its $1$-minor $\wt{\mu}$ is, and $|\mu| \le |\tau| \le m < 2n/\log n$. Hence, it suffices to prove the existence of a partition $\tau \nleqslant \lambda$ with $|\tau| \le m$.

Suppose, for sake of contradiction, that all partitions $\tau$ of size at most $m$ are minors of $\lambda$. Then, in particular, the partition
\[
\Bigg[ \underbrace{\left\lfloor \frac{m}{d} \right\rfloor, \dots, \left\lfloor \frac{m}{d} \right\rfloor}_{d} \Bigg]
\]
is a minor of $\lambda$ for $d = 1,2 ,\dots, m$. This shows that the $i$-th row of $\lambda$ has at least $\lfloor m/i \rfloor$ cells for all $i$. Therefore,
\begin{align*}
n &\ge m + \left\lfloor\frac{m}{2} \right\rfloor + \left\lfloor \frac{m}{3} \right\rfloor + \dots + \left\lfloor \frac{m}{m} \right\rfloor \\
&= m \cdot \left( 1 + \frac{1}{2} + \dots + \frac{1}{m} \right) - \left\{\frac{m}{2}\right\} - \left\{\frac{m}{3}\right\} - \dots - \left\{\frac{m}{m}\right\} \\
&\ge m \cdot \left( 1 + \frac{1}{2} + \dots + \frac{1}{m} \right) - \frac{1}{2} - \frac{2}{3} - \dots - \frac{m-1}{m} \\
&= (m+1) \cdot \left( 1 + \frac{1}{2} + \dots + \frac{1}{m} \right) - m \\
&> (m+1)(\log(m) + \gamma) - m \\
&= (m+1)(\log(m+1) - (1-\gamma)) + 1 - (m+1) \log\left(1+\frac{1}{m} \right), \tag{$\diamond$}
\end{align*}
where $\gamma \approx 0.5772$ (cf. \cite{OEIS16.001620}) is the Euler-Mascheroni constant.

Because $m = \lfloor 2n/\log n \rfloor$, we have the bound
\begin{align*}
(m+1) (\log(m+1) - (1-\gamma)) &> \frac{2n}{\log n} \cdot \big( \log n - \log \log n + \underbrace{\log 2 - (1-\gamma)}_{>0} \big) \\
&> 2n - \frac{2n \log \log n}{\log n}.
\end{align*}

Note that it is straightforward to check the numerical bound:
\[
\frac{\log \log n}{\log n} < \frac{1}{e} < 0.4
\]
for all integers $n \ge 2$. Therefore, we have 
\[
(m+1) (\log(m+1) - (1-\gamma)) > 1.2 n
\]
for $n \ge 2$. It is also straightforward to check that $(m+1) \log \left(1+\frac{1}{m}\right) < 1.4$ for all integers $m \ge 1$. Combining these inequalities with ($\diamond$), we have the inequality $n > 1.2 n - 0.4$, which contradicts our original assumption that $n \ge 2$.

Therefore, there must always be a partition $\mu$ of size at most $m$ that is not a minor of $\lambda$. This finishes our proof.
\end{proof}

\begin{remark}
In Lemma \ref{l:rectangle}, the constant $2$ in $\frac{2n}{\log n}$ can be further sharpened. For the purpose of this paper, however, the constant is not crucial to the asymptotic result in Theorem \ref{thm:Gnsimn} that $n-G(n) = O(n/\log n)$. Interested readers may calculate the optimal constant $C>0$ for which we can replace the bound of the size of $\mu$ in Lemma \ref{l:rectangle} with $|\mu| \le \frac{C \cdot n}{\log n}$.
\end{remark}

Now, we will prove Theorem \ref{thm:Gnsimn}.

\begin{center}
\begin{figure}
\begin{tikzpicture}
\begin{scope}[shift = {(0,0)}]
\ytableausetup{notabloids}
\ytableausetup{mathmode, boxsize=2.0em}
\node (n) {\ytableausetup{nosmalltableaux}
\ytableausetup{notabloids}
\ydiagram[*(gray!50)]{4+7,4+5,4+5,4+4,0,3,3,3,3,2,1}
*[*(gray!20)]{3+1,3+1,3+1,3+1,3}
*[*(white)]{3,3,3,3}};

\def \a {0.84};
\def \r {\a/4};

\draw[dashed, line width = 2pt] ({-9.5*\a +3.5},{1.0*\a-0.425}) -- ({-5.5*\a +3.5},{1.0*\a-0.425}) -- ({-5.5*\a +3.5},{6.0*\a-0.425});

\draw[decoration={brace,mirror,raise=5pt},decorate]
  ({-9.5*\a +3.5},{6.0*\a-0.425}) -- node[left=6pt] {$b$} ({-9.5*\a +3.5},{1.0*\a-0.425});

\draw[decoration={brace,mirror,raise=5pt},decorate]
  ({-5.5*\a +3.5},{6.0*\a-0.425}) -- node[above=6pt] {$a$} ({-9.5*\a +3.5},{6.0*\a-0.425});
\end{scope}

\def \a {0.84};

\begin{scope}[shift = {({-5.5*\a/6-4.5*\a+\a/12},{19*\a/12-32*\a/405})}]
\boat[]
\draw[->,thick] ({7*\a/6},{\a/3}) -- ({10*\a/6},{\a/3});
\end{scope}

\begin{scope}[shift = {({-11*\a/12-1.5*\a+32*\a/405},{5.5*\a})}, yscale = -1, rotate = 90]
\boat[]
\draw[->,thick] ({7*\a/6},{\a/3}) -- ({10*\a/6},{\a/3});
\end{scope}

\end{tikzpicture}
\caption{A schematic diagram for the proof of Theorem \ref{thm:Gnsimn}} \label{fig:doubleSONAR}
\end{figure}
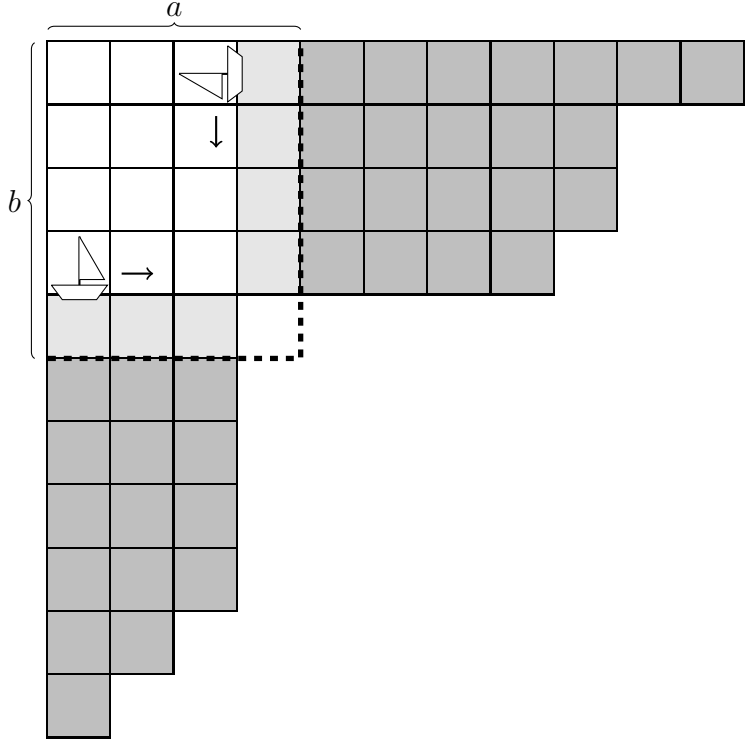
\end{center}

\begin{proof}[Proof of Theorem \ref{thm:Gnsimn}]
We will show that
\[
G(n) > n - \frac{2n}{\log n}
\]
for every integer $n \ge 2$, by proving that multiset-reconstructibility holds for the pair $\left(n, n- \lfloor 2n/\log n \rfloor \right)$. In other words, suppose $\lambda$ is an initially unknown partition of a known positive integer $n \ge 2$. We will show that if we know $N(\lambda/\mu)$ for all partitions $\mu$ with $|\mu| \le \lfloor 2n/\log n \rfloor$, then we can recover $\lambda$.

Our strategy is as follows. By Lemma \ref{l:rectangle}, there exists a rectangle partition $\rho = [a,a,\dots, a]$ of size at most $\lfloor 2n/\log n \rfloor$ with the property that $\rho$ is not a minor of $\lambda$, but its unique $1$-minor $\wt{\rho}$ is. We find such a partition explicitly by considering all the rectangle partitions $\rho$ of size at most $\lfloor 2n/\log n \rfloor$ and calculate $N(\lambda/\rho)$, for which we have the information. Then, we list all the rectangle partitions $\rho$ for which $N(\lambda/\rho)$ is zero. A minimal partition $\rho$ (with respect to $\le$) from the list will be the desired rectangle partition.

Next, let $\xi$ denote the rectangle partition whose width and height are both $1$ less than the respective side lengths of $\rho$. In Figure \ref{fig:doubleSONAR}, $\rho$ is shown as the rectangle partition bounded by the dashed segments, and $\xi$ is shown as the rectangle of white cells. The partition $\xi$ splits the cells of $\lambda$ into three parts: (1) the cells that are inside $\xi$, (2) the cells that are to the right of $\xi$, and (3) the cells that are below $\xi$. The cells to the right of $\xi$ form their own partition, which we denote $\lambda^{R}$. Analogously, the cells below $\xi$ also form their own partition, which we denote $\lambda^B$. Since $\xi$ is determined, in order to recover $\lambda$ it suffices to recover $\lambda^{R}$ and $\lambda^{B}$. To do so, we use SONAR. By taking $\mu$ to be a certain minor of $\wt{\rho}$, we will be able to obtain the sequences $\{E_m(\lambda^{B})\}_{m=0}^{\infty}$ and $\{E_m((\lambda^{R})^t)\}_{m=0}^{\infty}$, where $(\lambda^{R})^t$ denotes the conjugate of $\lambda^{R}$. Instead of doing single SONAR on $\lambda$, we will do double SONAR on the two partitions $\lambda^B$ and $(\lambda^{R})^t$ simultaneously. See Figure \ref{fig:doubleSONAR} for a schematic diagram for this method. As a result, we will recover $\lambda$.

As we described above, let $\rho = [a,a, \dots, a]$ be a rectangle partition of $b \ge 1$ parts with the property that $\wt{\rho} = [\underbrace{a, \dots, a}_{b-1}, a-1] \le \lambda$ but $\rho \nleqslant \lambda$ and that the size of $\rho$ is $ab \le \lfloor 2n/\log n \rfloor$. Following the description above, we let $\xi$ be the rectangle partition $[a-1, a-1, \dots, a-1]$ of $b-1$ parts. For convenience, we will use the notation $\rho^{u,v}$ to denote the partition
\[
\rho^{u,v} = \big[ \underbrace{a,\dots, a}_{u}, \underbrace{a-1, \dots, a-1}_{b-u-1},v \big]
\]
for non-negative integers $u \le b-1$ and $v \le a-1$. Note that 
\[
\xi = \rho^{0,0} \le \rho^{u,v} \le \rho^{b-1,a-1} = \wt{\rho}
\]
and therefore, the values $N(\lambda/\rho^{u,v})$ are known for all $u \le b-1$ and $v \le a-1$. Moreover, we know that $N(\lambda/\rho^{u,v}) >0$ for such $u$ and $v$.

Before doing double SONAR, let us consider a special case. If $b=1$, the partition $\wt{\rho}$ is actually the first row of $\lambda$ itself. This means that $\lambda_1 = a-1 \le \lfloor 2n/\log n\rfloor -1$. Therefore, not only can we recover $\lambda_1$, but we can recover $N(\lambda/[m])$ for all non-negative integers $m$. Consequently, we can recover the whole sequence $\{E_m(\lambda)\}_{m=0}^{\infty}$:
\[
E_m(\lambda) = \begin{cases}
\frac{N(\lambda/[m])}{N(\lambda)} \cdot \frac{n!}{(n-m)!} & \text{for } m \le a-1 \\
0 & \text{for } m \ge a.
\end{cases}
\]
In this case, of course, we do not need double SONAR, doing SONAR on $\lambda$ itself can recover the whole partition. Similarly, if $a=1$, we have that $(\lambda^t)_1 = b-1 \le \lfloor 2n/\log n \rfloor -1$. Doing SONAR on $\lambda^t$ can recover the whole $\lambda^t$. Therefore, we will assume from now on that $a,b \ge 2$. In particular, $\lambda^R$ and $\lambda^B$ are non-empty.

Let $|\lambda^R| = n^R$ and $|\lambda^B| = n^B$. We note that the integers $n^R$ and $n^B$ are initially unknown. Nevertheless, we know by counting the number of cells that
\[
n^R + n^B = n-(a-1)(b-1).
\]

The skew shape $\lambda/\rho^{u,v}$ consists of two separate skew shapes $\lambda^R/[u]^t$ and $\lambda^B/[v]$, of sizes $n^R-u$ and $n^B-v$, respectively. Therefore,
\begin{align*}
N(\lambda/\rho^{u,v}) &= \binom{n^R+n^B-u-v}{n^R-u} \cdot N(\lambda^R/[u]^t) N(\lambda^B/[v]) \\
&= \binom{n-(a-1)(b-1)-u-v}{n^R-u} \cdot N((\lambda^R)^t/[u]) N(\lambda^B/[v]).
\end{align*}
The left hand side $N(\lambda/\rho^{u,v})$ of the above equation is a known quantity. To use SONAR on $(\lambda^R)^t$ and $\lambda^B$, we want to determine $N((\lambda^R)^t/[u])$ and $N(\lambda^B/[v])$ for all non-negative integers $u$ and $v$. However, the binomial coefficient in the equation remains unknown at this point, because of the unknown $n^R$. The next step is to recover $n^R$ and $n^B$, so that the binomial coefficient becomes known.

Because $b \ge 2$, we can plug in $u=1$ in the formula above. Also, because $(\lambda^R)^t$ is non-empty, $N((\lambda^R)^t/[1]) = N((\lambda^R)^t/[0]) = N((\lambda^R)^t)$. Therefore,
\[
N(\lambda/\rho^{1,0}) = \binom{n-(a-1)(b-1)-1}{n^R-1} \cdot N((\lambda^R)^t) N(\lambda^B).
\]
This gives
\[
\frac{N(\lambda/\rho^{0,0})}{N(\lambda/\rho^{1,0})} = \frac{n-(a-1)(b-1)}{n^R}
\]
and therefore
\[
n^R = (n-(a-1)(b-1)) \cdot \frac{N(\lambda/\rho^{1,0})}{N(\lambda/\rho^{0,0})}.
\]
Since all the quantities in the right hand side of the equation above are known, we have recovered $n^R$, as well as $n^B$. As a result, we have also recovered
\[
N((\lambda^R)^t/[u]) \cdot N((\lambda^B)/[v])
\]
for all non-negative integers $u \le b-1$ and $v \le a-1$. In particular, setting $u=v=0$ implies that $N((\lambda^R)^t) \cdot N(\lambda^B)$ has been recovered. Therefore, the excitation factor
\[
E_u((\lambda^R)^t) = \frac{n^R!}{(n^R-u)!} \cdot \frac{ \left( N((\lambda^R)^t/[u]) \cdot N(\lambda^B/[0]) \right)}{ \left( N((\lambda^R)^t) \cdot N(\lambda^B) \right)}
\]
is also recovered for $0 \le u \le b-1$. For $u \ge b$, we already know that $E_u((\lambda^R)^t) = 0$. Hence, we have recovered the whole sequence $\{E_m((\lambda^R)^t)\}_{m=0}^{\infty}$. Similarly, the sequence $\{E_m(\lambda^B)\}_{m=0}^{\infty}$ is also recovered. Using SONAR with $(\lambda^R)^t$ and $\lambda^B$, we can recover the shapes of both $\lambda^R$ and $\lambda^B$, and therefore, the partition $\lambda$ can be reconstructed.

We have proved that for every integer $n \ge 2$,
\[
n - \frac{2n}{\log n} < G(n) \le n.
\]
Therefore, we conclude that $\lim_{n \rightarrow \infty} \frac{G(n)}{n} = 1$ with $n-G(n) = O(n/\log n)$ as $n \rightarrow \infty$ as desired.
\end{proof}

In the next section, we will look at the difference $n-G(n)$ in more detail.

\bigskip

\section{The Difference $n-G(n)$} \label{sec:n-G(n)}
In Section \ref{sec:asympGn}, we established a sublinear upper bound for $n-G(n)$. In this section, we will provide a lower bound for $n-G(n)$ and present computation results for a certain number of known values of $G(n)$. Computations of $G(n)$ become challenging as $n$ grows, because of the rapid growth rate of $p(n)$, the number of partitions of $n$. By a famous result due to Hardy and Ramanujan \cite{HR1918}, and independently due to Uspensky \cite{Uspen1920}, the function $p(n)$ has the asymptotic growth rate of
\[
p(n) \sim \frac{1}{4\sqrt{3} n} \cdot e^{\pi \sqrt{2n/3}}.
\]
Our experience shows that direct computations of $G(n)$ when $n>100$ is not a feasible task, unless one has a machine with high computational power.

\begin{prop} \label{p:Gnlen-2}
We have $G(0) = 0$, $G(1) = 1$, and for all positive integers $n \ge 2$, $G(n) \le n-2$.
\end{prop}

\begin{proof}
Because there is a unique partition of $n$ for $n=0,1$, we can always recover the unique partition. In these trivial cases, $G(0) = 0$ and $G(1) = 1$.

Now assume $n \ge 2$. Consider the partition $\lambda = [n]$. We have that $\wh{M}_{n-1}(\lambda) = \{1 \cdot [1]\} = \wh{M}_{n-1}(\lambda^t)$, but $\lambda \neq \lambda^t$. Thus, multiset-reconstructibility (MRC) fails for $(n,n-1)$, and hence $G(n) \le n-2$.
\end{proof}

Given a partition $\lambda$ of $n$, the following lemma gives a way to check whether $\wh{M}_{n-2}(\lambda) = \wh{M}_{n-2}(\lambda^t)$ using excitation factors.

\begin{lemma} \label{l:E2iffMn-2}
Let $\lambda$ be a partition of $n \ge 2$. Then, $\wh{M}_{n-2}(\lambda) = \wh{M}_{n-2}(\lambda^t)$ if and only if $E_2(\lambda) = E_2(\lambda^t)$.
\end{lemma}
\begin{proof}
Observe that $\wh{M}_{n-2}(\lambda) = \wh{M}_{n-2}(\lambda^t)$ if and only if $N(\lambda/[2]) = N(\lambda^t/[2])$. By Naruse's Skew-Shape Hook Length Formula (Theorem \ref{thm:NaruseSSHLF}), $N(\lambda/[2]) = N(\lambda^t/[2])$ if and only if
\[
\frac{N(\lambda)}{n(n-1)} \cdot E_2(\lambda) = \frac{N(\lambda)}{n(n-1)} \cdot E_2(\lambda^t)
\]
that is, if and only if $E_2(\lambda) = E_2(\lambda^t)$.
\end{proof}

\begin{center}
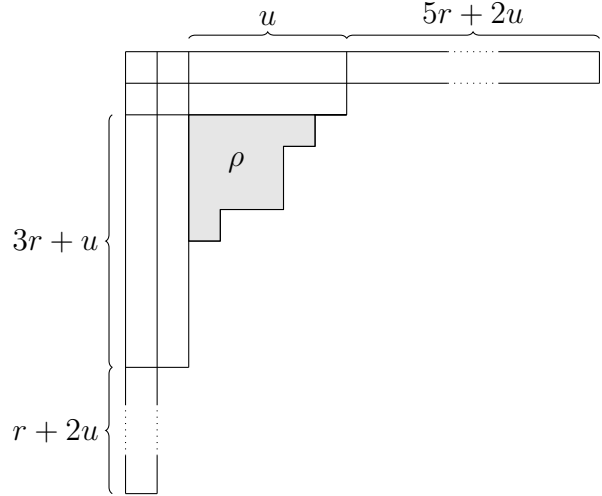
\begin{figure}
\begin{tikzpicture}
\def \a {0.42};

\draw[fill = gray!20] ({2*\a},{8*\a}) -- ({2*\a},{4*\a}) -- ({3*\a},{4*\a}) -- ({3*\a},{5*\a}) -- ({5*\a},{5*\a}) -- ({5*\a},{7*\a}) -- ({6*\a},{7*\a}) -- ({6*\a},{8*\a}) -- ({7*\a},{8*\a}) -- ({2*\a},{8*\a});

\draw (0,{-4*\a}) -- (0,{-2.8*\a});
\draw[dotted] (0,{-2.8*\a}) -- (0,{-1.2*\a});
\draw (0,{-1.2*\a}) -- (0,{10*\a});
\draw ({1*\a},{-4*\a}) -- ({1*\a}, {-2.8*\a});
\draw[dotted] ({1*\a},{-2.8*\a}) -- ({1*\a}, {-1.2*\a});
\draw ({1*\a},{-1.2*\a}) -- ({1*\a}, {10*\a});
\draw ({2*\a},{0*\a}) -- ({2*\a}, {10*\a});
\draw ({3*\a},{4*\a}) -- ({3*\a}, {5*\a});
\draw ({5*\a},{5*\a}) -- ({5*\a}, {7*\a});
\draw ({6*\a},{7*\a}) -- ({6*\a}, {8*\a});
\draw ({7*\a},{8*\a}) -- ({7*\a}, {10*\a});
\draw ({15*\a},{9*\a}) -- ({15*\a}, {10*\a});

\draw ({0*\a},{10*\a}) -- ({10.2*\a}, {10*\a});
\draw[dotted] ({10.2*\a},{10*\a}) -- ({11.8*\a}, {10*\a});
\draw ({11.8*\a},{10*\a}) -- ({15*\a}, {10*\a});
\draw ({0*\a},{9*\a}) -- ({10.2*\a}, {9*\a});
\draw[dotted] ({10.2*\a},{9*\a}) -- ({11.8*\a}, {9*\a});
\draw ({11.8*\a},{9*\a}) -- ({15*\a}, {9*\a});
\draw ({0*\a},{8*\a}) -- ({7*\a}, {8*\a});
\draw ({5*\a},{7*\a}) -- ({6*\a}, {7*\a});
\draw ({3*\a},{5*\a}) -- ({5*\a}, {5*\a});
\draw ({2*\a},{4*\a}) -- ({3*\a}, {4*\a});
\draw ({0*\a},{0*\a}) -- ({2*\a}, {0*\a});
\draw ({0*\a},{-4*\a}) -- ({1*\a}, {-4*\a});

\node at ({3.5*\a},{6.5*\a}) {$\rho$};

\draw[decoration={brace,mirror,raise=5pt},decorate]
  ({0*\a},{0*\a}) -- node[left=6pt] {$r+2u$} ({0*\a},{-4*\a});
\draw[decoration={brace,mirror,raise=5pt},decorate]
  ({0*\a},{8*\a}) -- node[left=6pt] {$3r+u$} ({0*\a},{0*\a});

\draw[decoration={brace,mirror,raise=5pt},decorate]
  ({7*\a},{10*\a}) -- node[above=6pt] {$u$} ({2*\a},{10*\a});
\draw[decoration={brace,mirror,raise=5pt},decorate]
  ({15*\a},{10*\a}) -- node[above=6pt] {$5r+2u$} ({7*\a},{10*\a});
\end{tikzpicture}
\caption{The partition $\lambda^{r,u,\rho}$ in the proof of Proposition \ref{p:Gnisn-2}. This picture shows the partition for the case when $r=1$, $u=5$, and $\rho = [4,3,3,1]$. (In this figure, we shrink the arm and the leg of the largest hook where the dotted segments are shown.)} \label{fig:lambdarurho}
\end{figure}
\end{center}

It turns out that $n-G(n) = 2$ for only finitely many values of $n$, as shown in the following proposition.

\begin{prop} \label{p:Gnisn-2}
Let $n$ be a non-negative integer. Then, $G(n) = n-2$ if and only if $2 \le n \le 11$ or $n=13$.
\end{prop}
\begin{proof}
To show that $G(n) = n-2$ for $2 \le n \le 11$ or $n=13$, we note that the size of partition $n$ is small enough that calculations can be done by hand, or we may use computer search. Namely, for each such $n$, we list all the partitions $\lambda \vdash n$ and calculate the pair $\left( N(\lambda/[2]), N(\lambda/[1,1]) \right)$, for every $\lambda \vdash n$. Then, we verify directly that no two pairs are identical. This calculation does not take much time, as even for the largest case, in which $n=13$, the number of partitions is $p(13) = 101$ (cf. \cite{OEIS16.000041}).

To show that $G(n) \le n-3$ for $n=12$ or $n \ge 14$, we explicitly construct a pair of different partitions $\lambda$ and $\mu$ for each $n$ so that $\wh{M}_{n-2}(\lambda) = \wh{M}_{n-2}(\mu)$. This task can be done by computer search for $n=12$ or $14 \le n \le 50$. For $n \ge 51$, we claim that there exists a partition $\lambda \vdash n$ which is not self-conjugate ($\lambda \neq \lambda^t$) such that
\[
\wh{M}_{n-2}(\lambda) = \wh{M}_{n-2}(\lambda^t). \tag{$\heartsuit$}
\]

For integers $r \ge 1$ and $u \ge 0$, consider the following partition
\[
\lambda^{r,u} := \big[ 5r+3u+2, u+2, \underbrace{2, \dots, 2}_{3r+u}, \underbrace{1, \dots, 1}_{r+2u} \big].
\]
Note that because $r \ge 1$, $\lambda^{r,u}$ is not self-conjugate. We use the formula in Corollary \ref{cor:EmFormula} to calculate the excitation factor $E_2(\lambda^{r,u})$ directly:
\begin{align*}
E_2(\lambda^{r,u}) &= (9r+6u+3)(8r+4u+2)+(9r+6u+3)u+(3r+2u+1)u \\
&= 72r^2+32u^2+96ru+42r+28u+6.
\end{align*}
Similarly,
\begin{align*}
E_2((\lambda^{r,u})^t) &= (9r+6u+3)(4r+4u+2)+(9r+6u+3)(3r+u) \\
& \hphantom{=} +(3r+2u+1)(3r+u) = 72r^2+32u^2+96ru+42r+28u+6.
\end{align*}
This shows that $E_2(\lambda^{r,u}) = E_2((\lambda^{r,u})^t)$, and therefore, by Lemma \ref{l:E2iffMn-2},
\[
\wh{M}_{n-2}(\lambda^{r,u}) = \wh{M}_{n-2}((\lambda^{r,u})^t).
\]
Note that $\lambda$ is a partition of $n = 4(3r+2u+1)$ where $r \ge 1$ and $u \ge 0$. This proves the claim in ($\heartsuit$) for all positive integers $n$ such that $n \ge 24$ and $n \equiv 0$ modulo $4$. 

It turns out that for $n \not\equiv 0$ modulo $4$, we can construct the desired partition $\lambda \vdash n$ by modifying $\lambda^{r,s}$ as follows. Let $\rho$ be any self-conjugate partition such that 
\[
\rho \le \big[ \underbrace{u, \dots, u}_u \big].
\]
We define the partition $\lambda^{r,u,\rho}$ (cf. Figure \ref{fig:lambdarurho}.) as
\[
\lambda^{r,u,\rho} := \big[ 5r+3u+2, u+2, 2+\rho_1, 2+\rho_2, \dots, 2+\rho_u, \underbrace{2,\dots, 2}_{3r}, \underbrace{1, \dots, 1}_{r+2u} \big].
\]
A nice property of this construction is that adding $\rho$ does not change the value of the function $E_2(\bullet) - E_2(\bullet^t)$. To see that, we note as a result of Corollary \ref{cor:EmFormula} that, for any partition $\tau$, we have
\[
E_2(\tau) - E_2(\tau^t) = \sum_{i=1}^{\rk \tau} h_{ii}(\tau) \cdot \left( \tau_i - (\tau^t)_i \right)
\]
where $\rk \tau$ denotes the greatest integer $m$ for which $\big[ \underbrace{m,\dots,m}_m \big] \le \tau$ and $h_{ij}(\tau)$ denotes the hook length of the cell $(i,j)$ in $\tau$. Note that $\tau_i - (\tau^t)_i$ is the arm length minus the leg length of the hook at $(i,i)$. Therefore, each summand in the summation above only depends on the hook shape of the hook. In particular, if the hook is self-conjugate, the corresponding summand is zero. Hence, adding a self-conjugate partition $\rho$ to the lower-right part of the Young diagram of the partition $\lambda^{r,u}$ as in Figure \ref{fig:lambdarurho} does not change $E_2(\bullet) - E_2(\bullet^t)$. Therefore,
\[
E_2(\lambda^{r,u,\rho}) - E_2((\lambda^{r,u,\rho})^t) = 0
\]
and the partition serves as the desired example for $n = 4(3r+2u+1)+|\rho|$. Since $\rho$ can be chosen to be any self-conjugate minor of $[u,\dots, u]$, we can choose $|\rho|$ to be any non-negative integer not exceeding $u^2$ except $2$ and $u^2-2$. In particular, if we consider $u \ge 3$, we can choose $|\rho|$ to be any integer from $\{3,4,5,6\}$, and therefore,
\[
n= 4\Big( 10 + 3\underbrace{(r-1)}_{\in \mathbb{Z}_{\ge 0}} + 2\underbrace{(u-3)}_{\in \mathbb{Z}_{\ge 0}} \Big) + |\rho|
\]
can be any integer $n \ge 51$. This finishes the proof.
\end{proof}

Now that we have determined all $n$ for which $G(n) = n-2$, the characterization of all $n$ for which $G(n) = n-3$ is not difficult. We have the following lemma.

\begin{lemma} \label{l:n-3forfree}
Let $\lambda$ be a partition of an integer $n \ge 3$. Then, $\wh{M}_{n-2}(\lambda) = \wh{M}_{n-2}(\lambda^t)$ if and only if $\wh{M}_{n-3}(\lambda) = \wh{M}_{n-3}(\lambda^t)$.
\end{lemma}
\begin{proof}
($\Leftarrow$) If $\wh{M}_{n-3}(\lambda) = \wh{M}_{n-3}(\lambda^t)$, then $\wh{M}_{n-2}(\lambda) = \wh{M}_{n-2}(\lambda^t)$ follows from Proposition \ref{p:kthenk'}.

($\Rightarrow$) Suppose $\wh{M}_{n-2}(\lambda) = \wh{M}_{n-2}(\lambda^t)$. Then,
\[
N(\lambda/[2]) = N(\lambda^t/[2]) = N(\lambda/[1,1]).
\]
Therefore,
\begin{align*}
N(\lambda/[3]) &= N(\lambda/[2]) - N(\lambda/[2,1]) \\
&= N(\lambda/[1,1]) - N(\lambda/[2,1]) = N(\lambda/[1,1,1]) = N(\lambda^t/[3]).
\end{align*}
The equality $N(\lambda/[1,1,1]) = N(\lambda^t/[1,1,1])$ is obtained similarly. We also have $N(\lambda/[2,1]) = N(\lambda^t/[2,1])$ because $[2,1]$ is self-conjugate. Therefore, $\wh{M}_{n-3}(\lambda) = \wh{M}_{n-3}(\lambda^t)$.
\end{proof}

Lemma \ref{l:n-3forfree} shows that the infinite family of examples of partitions $\lambda$ such that $\lambda \neq \lambda^t$ and $\wh{M}_{n-2}(\lambda) = \wh{M}_{n-2}(\lambda^t)$ in the proof of Proposition \ref{p:Gnisn-2} also gives $\wh{M}_{n-3}(\lambda) = \wh{M}_{n-3}(\lambda^t)$ for all $n \ge 51$. For $n \le 50$, we do computer search and find that for $n \in \{12,14,17,18,23\}$, there are no pairs $(\lambda, \mu)$ of different partitions $\lambda, \mu \vdash n$ for which $\wh{M}_{n-3}(\lambda) = \wh{M}_{n-3}(\mu)$. Thus, we have the following results.

\begin{prop} \label{p:Gnisn-3}
Let $n$ be a non-negative integer. Then, $G(n) = n-3$ if and only if $n \in \{12,14,17,18,23\}$.
\end{prop}

\begin{cor}
For every integer $n \ge 24$, we have $G(n) \le n-4$.
\end{cor}

\begin{remark} \label{rmk:PixtonsComSearch}
A computer search by Aaron Pixton \cite{Pix16} shows that there are no pairs $(\lambda, \mu)$ of different partitions of $n$ such that $\wh{M}_{n-4}(\lambda) = \wh{M}_{n-4}(\mu)$ for all integers $n \le 59$. Therefore, $G(n) = n-4$ for $24 \le n \le 59$ and for $n \in \{15,16,19,20,21,22\}$. For $n=60$, however, Pixton found that the pair
\[
\lambda = [14,11,8,6,3,3,3,2,2,2,2,2,2]
\]
and
\[
\mu = [13,13,7,4,4,4,3,3,2,2,2,1,1,1]
\]
satisfies $\wh{M}_{n-4}(\lambda) = \wh{M}_{n-4}(\mu)$. Surprisingly, this pair also satisfies $\wh{M}_{n-5}(\lambda) = \wh{M}_{n-5}(\mu)$. Pixton observed that for $n \le 81$, there are no two different partitions with the same multiset of $(n-6)$-minors. Therefore, $G(60) = 60-6=54$.
\end{remark}

\begin{remark} \label{rmk:ComSearchTable}
The following table, due to Aaron Pixton \cite{Pix16}, shows the values of $n-G(n)$ for $60 \le n \le 70$.
\begin{center}
\begin{tabular}{ |c|c|c|c|c|c|c|c|c|c|c|c|c|c|c|c|c|c|c|c|c| } 
\hline
$n$ & $60$ & $61$ & $62$ & $63$ & $64$ & $65$ & $66$ & $67$ & $68$ & $69$ & $70$ \\
\hline
$n-G(n)$ & $6$ & $4$ & $4$ & $6$ & $6$ & $6$ & $4$ & $6$ & $6$ & $6$ & $6$ \\
\hline
\end{tabular}
\end{center}
Pixton also found that $G(n) = n-6$ for all $71 \le n \le 81$.
\end{remark}

\bigskip

\section{Further Research Possibilities}
In this final section, we suggest a few ideas and questions which provide directions for further studies.

\medskip

\textbf{1. How large can $n-G(n)$ be?} Our main result, Theorem \ref{thm:Gnsimn} shows that $n-G(n) = O(n/ \log n)$. However, it is not clear to us whether $n-G(n)$ is bounded above by a constant. We conjecture that it has no finite upper bound.

\begin{conj} \label{conj:difftoinfty}
The difference $n-G(n)$ satisfies
\[
\lim_{n \rightarrow \infty} n-G(n) = + \infty.
\]
\end{conj}

In Section \ref{sec:n-G(n)}, we show that $n-G(n) \ge 3$ for $n \ge 14$ and $n-G(n) \ge 4$ for $n \ge 24$, by constructing an infinite family of pairs $(\lambda,\mu)$ of different partitions of $n$ such that $\wh{M}_{n-2}(\lambda) = \wh{M}_{n-2}(\mu)$ for all sufficiently large $n$. For every positive integer $k$, if one can explicitly construct an infinite family of pairs $(\lambda, \mu)$ of different partitions of $n$ such that $\wh{M}_{n-k}(\lambda) = \wh{M}_{n-k}(\mu)$ for all sufficiently large $n$, then Conjecture \ref{conj:difftoinfty} will follow.

\medskip

\textbf{2. The set $\{n-G(n)|n \in \mathbb{Z}_{\ge 0}\}$.} We have observed that $n-G(n)$ can be $0$, $2$, $3$, $4$, and $6$. Proposition \ref{p:Gnlen-2} shows that there are no integers $n$ for which $n-G(n) = 1$. Pixton \cite{Pix16} gives all the values of $G(n)$ up to $n = 81$. However, we do not know whether there exists $n$ for which $G(n) = n-5$. In general, what are all the non-negative integers that do not belong to the set $\{n-G(n)|n \in \mathbb{Z}_{\ge 0}\}$?

\medskip

\textbf{3. Why is $E_m(\lambda)$ a $\mathbb{Q}[k]$-linear combination of $\sigma_i\left(a_1, \dots, a_k\right)$?} Theorem \ref{thm:Qklincomb} appears to us as a surprising phenomenon on which the SONAR technique is based. Specifically, it is intriguing that $E_m(\lambda)$ is symmetric in $a_1, \dots, a_k$. Although we have a proof for the theorem, we would still like to find a more intuitive explanation for why this phenomenon occurs.

\bigskip

\section*{Acknowledgments}
This research was done at the University of Minnesota Duluth in the summer of 2016, with the financial support from the National Science Foundation (grant number: NSF-1358659) and the National Security Agency (grant number: NSA H98230-16-1-0026). 

I would like to thank Joe Gallian for hosting me at the university, introducing me to the partition multiset-reconstruction problem, and proofreading an earlier version of this paper. I am particularly grateful to Maria Monks for giving me suggestions and crucial insights on the problem and for commenting on an earlier version of this paper. 

The coefficients in the formula of Theorem \ref{thm:Qklincomb} were once described using polynomial recursion. Benjamin Gunby pointed out to me that they can be described using the Stirling polynomials. This observation resulted in the current version of the formula. I would like to thank him for sharing with me this important insight. In addition, I would like to thank Levent Alpoge and Mitchell M. Lee for their helpful ideas. I would like to express my gratitude to Aaron Pixton for his interest in this research project, his helpful suggestions, especially in Remark \ref{rmk:PixtonsFormulas}, and his powerful computational results, especially those which lead to Remarks \ref{rmk:PixtonsComSearch} and \ref{rmk:ComSearchTable}. I am appreciative of Eric Riedl for giving me Algebraic Geometry insights related to the proof of Lemma \ref{l:polyptscheck}.

Additionally, I would like to thank Chantra Wangcharoenwong for helping me on designing objects in Figures \ref{fig:SONAR} and \ref{fig:doubleSONAR}. I am also thankful for my brother Punyawut Jiradilok for his comments and suggestions on the aesthetic qualities of figures in this paper.

My undergraduate studies at Harvard University are supported by King's Scholarship (Thailand).


\bigskip

\begin{thebibliography}{99}

\bibitem{EMOT81} A.~Erd\'elyi, W.~Magnus, F.~Oberhettinger, and F.~G.~Tricomi, Higher Transcendental Functions, Vol. 3. Krieger, New York, 1981.

\bibitem{FRT54} J.~S.~Frame, G.~de~B.~Robinson, and R.~M.~Thrall, The hook graphs of the symmetric group, {\em Canad. J. Math} 6.316 (1954): C324.

\bibitem{Ful97} W.~Fulton, Young tableaux: with applications to representation theory and geometry. Vol. 35. Cambridge University Press, 1997.

\bibitem{HR1918} G.~H.~Hardy, and S.~Ramanujan, Asymptotic formulae in the theory of partitions. {\em Proc. London Math. Soc.} 2.17 (1918): 75-115.

\bibitem{JK81} G.~James, and A.~Kerber. The Representation Theory of the Symmetric Group. Reading, Mass. Addison-Wesley, 1981.

\bibitem{Mo09} M.~Monks, The solution to the partition reconstruction problem, {\em J. Comb. Theory. A.} 116 (2009), 76-91.

\bibitem{MPP15} A.~Morales, I.~Pak, and G.~Panova, Hook formulas for skew shapes, \texttt{arXiv:1512.08348} (2015).

\bibitem{Na14} H.~Naruse, Schubert calculus and hook formula, talk slides at 73rd S´em. Lothar. Combin., Strobl, Austria, 2014.

\bibitem{OEIS16.000041} OEIS Foundation Inc. (2016), The On-Line Encyclopedia of Integer Sequences, https://oeis.org/A000041.

\bibitem{OEIS16.001620} OEIS Foundation Inc. (2016), The On-Line Encyclopedia of Integer Sequences, https://oeis.org/A001620.

\bibitem{Pix16} A.~Pixton, Private communication, July 2016.

\bibitem{PS05} O.~Pretzel, and J.~Siemons, Reconstruction of partitions, {\em Electron. J. Combin.} 11.2 (2005): N5.

\bibitem{Ro05} S.~Roman, \enquote{The Stirling Polynomials.} \S 4.8 in The Umbral Calculus. New York: Academic Press, 1984.

\bibitem{Sa13} B.~Sagan, The symmetric group: representations, combinatorial algorithms, and symmetric functions. Vol. 203. Springer Science \& Business Media, 2013.

\bibitem{Stan13} R.~P.~Stanley, \enquote{Chapter 7: Enumeration Under Group Action.} In Algebraic combinatorics walks, trees, tableaux, and more. Springer, New York, 2013.

\bibitem{Stan99} R.~P.~Stanley, Enumerative Combinatorics, vol. 2. Cambridge Univ. Press, Cambridge, 1999.

\bibitem{Uspen1920} J.~V.~Uspensky, Asymptotic formulae for numerical functions which occur in the theory of partitions, {\em Bull. Russ. Acad. Sci.} 14.6 (1920): 199-218.

\bibitem{Vat08} V. Vatter, A sharp bound for the reconstruction of partitions, {\em Electron. J. Combin.} 15 (2008), N23.

\end{thebibliography}
\end{document}